\documentclass[a4paper,10pt,draft]{amsart}

\usepackage{verbatim, amssymb, enumerate}

\renewcommand \a{\alpha}
\renewcommand \b{\beta}

\newcommand \K{\delta}

\newcommand \n{\nabla}
\newcommand \la{\lambda}
\newcommand \id{\operatorname{id}}
\newcommand \br{\mathbb{R}}
\newcommand \bc{\mathbb{C}}
\newcommand \Oc{\mathbb{O}}
\newcommand \Rn{\mathbb R^n}
\newcommand \Ro{\mathbb R^8}
\newcommand \Rs{\mathbb R^{16}}
\newcommand \Rd{\mathbb R^9}

\newcommand \rk{\operatorname{rk}}
\newcommand \Ker{\operatorname{Ker}}
\newcommand \End{\operatorname{End}}
\newcommand \Sk{\operatorname{Skew}}
\newcommand \Sym{\operatorname{Sym}}
\renewcommand \Im{\mathrm{Im}}

\newcommand \Cliff{\operatorname{Cliff}}
\newcommand \Cl{\operatorname{Cl}}
\newcommand \Span{\operatorname{Span}}
\newcommand \Tr{\operatorname{Tr}}

\newcommand \cJ{\mathcal{J}}
\newcommand \cI{\mathcal{I}}
\newcommand \cp{\mathcal{C}}
\newcommand \mU{\mathcal{U}}
\newcommand \mL{\mathcal{L}}
\newcommand \mA{\mathcal{A}}
\newcommand \mS{\mathcal{S}}
\newcommand \ve{\varepsilon}

\newcommand \OC{\Oc \otimes \bc}
\newcommand \<{\langle}
\renewcommand \>{\rangle}
\newcommand \JC{\mathcal{J}_{\bc}}
\newcommand \IC{\mathcal{I}_{\bc}}

\theoremstyle{plane}
\newtheorem{theorem}{Theorem}
\newtheorem*{theorem*}{Theorem}
\newtheorem*{conjecture*}{Conjecture}
\newtheorem*{corollary*}{Corollary}
\newtheorem{lemma}{Lemma}

\newtheorem*{namedtheorem}{\theoremname}
\newcommand{\theoremname}{te}
\newenvironment{named}[1]{\renewcommand{\theoremname}{#1}
\begin{namedtheorem}}{\end{namedtheorem}}

\theoremstyle{definition}
\newtheorem{definition}{Definition}

\theoremstyle{remark}
\newtheorem{remark}{Remark}

\begin{document}

\title[Osserman manifolds and Weyl-Schouten Theorem for rank-one symmetric spaces]
{Osserman manifolds and the Weyl-Schouten Theorem for rank-one symmetric spaces}

\author{Y.Nikolayevsky}
\address{Department of Mathematics and Statistics, La Trobe University, Victoria, 3086, Australia}
\email{y.nikolayevsky@latrobe.edu.au}

\date{\today}


\begin{abstract}
A Riemannian manifold is called Osserman (conformally Osserman, respectively), if the eigenvalues of the Jacobi
operator of its curvature tensor (Weyl tensor, respectively) are constant on the unit tangent sphere at every point.
Osserman Conjecture asserts that every Osserman manifold is either flat or rank-one symmetric.
We prove that both the Osserman Conjecture and its conformal version, the Conformal Osserman Conjecture, are true,
modulo a certain assumption on algebraic curvature tensors in $\Rs$. As a consequence, we show that a Riemannian
manifold having the same Weyl tensor as a rank-one symmetric space, is conformally equivalent to it.
\end{abstract}

\maketitle

\section{Introduction}
\label{s:intro}

The aim of this paper is twofold. Firstly, we consider Osserman and conformally Osserman manifolds of dimension $16$
(which is the only missing dimension in the proof of the Osserman Conjecture). We show that both the ``genuine"
Osserman
Conjecture and its conformal version can be reduced to a purely algebraic question on algebraic curvature tensors
in $\Rs$. Secondly, we obtain an analogue of the classical Weyl-Schouten Theorem for rank-one symmetric spaces: a
Riemannian manifold of dimension greater than four having ``the same" Weyl tensor as that of a rank-one
symmetric space is locally conformally equivalent to that space.

An \emph{algebraic curvature tensor} $\mathcal{R}$ on a Euclidean space $\Rn$ is a $(3, 1)$ tensor having the same
symmetries as the curvature tensor of a Riemannian manifold. For $X \in \Rn$, the \emph{Jacobi operator}
$\mathcal{R}_X : \Rn \to \Rn$ is defined by $\mathcal{R}_XY = \mathcal{R}(X, Y)X$ . The Jacobi operator is symmetric
and $\mathcal{R}_XX = 0$ for all $X \in \Rn$.

\begin{definition} \label{d:oact}
An algebraic curvature tensor $\mathcal{R}$ is called \emph{Osserman} if the eigenvalues of the Jacobi operator
$\mathcal{R}_X$ do not depend on the choice of a unit vector $X \in \Rn$.
\end{definition}

One of the algebraic curvature tensors naturally associated to a Riemannian manifold (apart from the curvature tensor
itself) is the Weyl conformal curvature tensor.

\begin{definition} \label{d:com}
A Riemannian manifold is called \emph{(pointwise) Osserman} if its curvature tensor at every point is Osserman. A
Riemannian manifold is called \emph{conformally Osserman} if its Weyl tensor at every point is Osserman.
\end{definition}

It is well-known (and is easy to check directly) that a Riemannian space locally isometric to a Euclidean space or to
a rank-one symmetric space is Osserman. The question of whether the converse is true is known as the Osserman
Conjecture \cite{O}:

\begin{named}{Osserman Conjecture}
Any smooth pointwise Osserman manifold of dimension $n \ne 2,4$ is either flat or locally rank-one symmetric.
\end{named}

The study of conformally Osserman manifolds was started in \cite{BG1} and then continued in
\cite{BG2,BGNSi,Gil,BGNSt}. Every Osserman manifold is conformally Osserman (which easily follows from the formula for
the Weyl tensor and the fact that every Osserman manifold is Einstein), as also is every manifold locally conformally
equivalent to an Osserman manifold. This motivates the following conjecture made in \cite{BGNSi}: 

\begin{named}{Conformal Osserman Conjecture}
Any smooth conformally Osserman manifold of dimension $n > 4$ is either conformally flat or locally conformally
equivalent to a rank-one symmetric space.
\end{named}

The proof of the Osserman Conjecture for manifolds of dimension not divisible by $4$ was given in \cite{Chi}, before
the conjecture itself was published. The Conformal Osserman Conjecture for manifolds of dimension $n > 6$, not
divisible by $4$, is proved in \cite{BG1}, 
for manifolds with the structure of a warped product, both conjectures are proved in \cite{BGV}.

At present, both the Osserman Conjecture and the Conformal Osserman Conjecture are proved in all the cases, with the
only exception when $n=16$ and one of the eigenvalues of the Jacobi operator
has multiplicity $7$ or $8$ \cite{Nhjm,Nmm,Nma,Nbel,Nco}.
The main difficulty in this remaining case lies in the following algebraic fact: it can
be shown that in all the other cases, an Osserman algebraic curvature tensor has a
\emph{Clifford structure}, so it ``looks similar" to the curvature tensor of the complex or the quaternionic projective
space (a Clifford structure arises from an orthogonal representation of a Clifford algebra; see
Section~\ref{s:clifford} for details). However, the curvature
tensor of the Cayley projective plane (whose Jacobi operator has eigenvalues with multiplicities exactly $7$ and $8$)
is essentially different. This is the only known Osserman curvature tensor without a Clifford structure, which
motivates the following algebraic conjecture.

\begin{named}{Conjecture~A}
Every Osserman algebraic curvature tensor in $\Rs$ whose Jacobi operator has an eigenvalue of
multiplicity $7$ or of multiplicity $8$ either has a Clifford structure or is a linear combination of the constant
curvature tensor and the curvature tensor of the Cayley projective plane.
\end{named}

In the latter case, we will say that an algebraic curvature tensor has a \emph{Cayley structure}.

Our main result is the following theorem.
\begin{theorem} \label{t:co}
Assuming Conjecture A, both the Osserman Conjecture and the Conformal Osserman Conjecture are true.
\end{theorem}

As a consequence of Theorem~\ref{t:co}, we obtain the following analogue of the Weyl-Schouten Theorem for
rank-one symmetric spaces (without assuming Conjecture~A):

\begin{theorem}\label{t:ws}
Suppose that for every point $x$ of a  smooth Riemannian manifold $M^n, \; n > 4$, there exists a linear isometry
which maps the Weyl tensor of $M^n$ at $x$ on a positive multiple of the Weyl tensor of a rank-one symmetric
space $M_0^n$. Then $M^n$ is locally conformally equivalent to $M_0^n$.
\end{theorem}

Manifolds satisfying the assumption of Theorem~\ref{t:ws} can be viewed as conformal analogues of (a subclass of)
curvature homogeneous manifolds \cite{TV,Gil}. 

In dimension four, both theorems are false. For Theorem~\ref{t:ws} and the conformal part
of Theorem~\ref{t:co}, this follows from the fact that a four-dimensional Riemannian manifold is conformally Osserman
if and only if it is either self-dual or anti-self-dual \cite{BG2} and that there exist self-dual K\"{a}hler metrics
on $\bc^2$ which are not locally conformally equivalent to locally-symmetric ones \cite{Der}. For the ``genuine"
Osserman part of Theorem~\ref{t:co}, the counterexample is given by the generalized complex space forms
\cite[Corollary 2.7]{GSV}.
However, the Osserman Conjecture is true for four-dimensional \emph{globally} Osserman manifolds, that is, for those
whose Jacobi operator has constant eigenvalues on the whole unit tangent bundle \cite{Chi}.

\smallskip

The paper is organized as follows. Section~\ref{s:clifford} gives the algebraic background for the proof of the both 
theorems: we consider Osserman algebraic curvature tensors on $\Rs$, each of which, assuming Conjecture~A, has either 
a Clifford structure (discussed in Sections~\ref{ss:clifford},\ref{ss:cs16}), or a Cayley structure 
(Section~\ref{ss:actcayley}). The proofs of the both theorems are given in Section~\ref{s:conf}. We first prove the 
local version of the conformal part of Theorem~\ref{t:co} using the second Bianchi identity, separately in the
Clifford case (Section~\ref{ss:cc}) and in the Cayley case (Section~\ref{ss:oc}), and then the global version, by
showing that the ``algebraic type" of the Weyl tensor is the same at all the points of a connected conformally 
Osserman Riemannian manifold (Section~\ref{ss:glo}). The proofs of Theorem~\ref{t:ws} and of the ``genuine" Osserman 
part of Theorem~\ref{t:co} easily follow (see the second and the last paragraphs of Section~\ref{s:conf}, 
respectively).

The Riemannian manifold $M^n$ is assumed to be smooth (of class $C^\infty$), although both theorems remain valid for
class $C^k$, with sufficiently large $k$.

\section{Osserman algebraic curvature tensors in $\Rs$ and the Clifford structure}
\label{s:clifford}

Both Theorem~\ref{t:co} and Theorem~\ref{t:ws} have to be proved only when $n=16$ (see Section~\ref{s:intro}).
In this section, we consider all the known Osserman algebraic curvature tensors in $\Rs$, namely, the algebraic 
curvature tensors with a Clifford structure and the algebraic curvature tensors with a Cayley structure.

\subsection{Clifford structure}
\label{ss:clifford}

The property of an algebraic curvature tensor $\mathcal{R}$ to be Osserman is quite algebraically restrictive. In the
most cases, such a tensor can be obtained by the following construction, suggested in \cite{GSV}, which
generalizes the curvature tensors of the complex and the quaternionic projective spaces.

\begin{definition}\label{d:cl}
A \emph{Clifford structure} $\Cliff(\nu ; J_1, \dots, J_\nu; \la_0, \eta_1, \dots , \eta_\nu)$ on a
Euclidean space $\Rn$ is a set of $\nu \ge 0$ anticommuting almost Hermitian structures $J_i$ and $\nu+1$ real numbers
$\la_0, \eta_1, \dots \eta_\nu$, with $\eta_i \ne 0$. An algebraic curvature tensor $\mathcal{R}$ on $\Rn$ \emph{has a
Clifford structure} $\Cliff(\nu ; J_1, \dots, J_\nu; \la_0, \eta_1, \dots , \eta_\nu)$ if
\begin{equation}\label{eq:cs}
\mathcal{R}(X, Y) Z = \la_0 (\< X, Z \> Y - \< Y, Z \> X) \\
+ \sum\nolimits_{i=1}^\nu \eta_i (2 \< J_iX, Y \> J_iZ + \< J_iZ, Y \> J_iX - \< J_iZ, X \> J_iY).
\end{equation}
When this does not create ambiguity, we abbreviate
$\Cliff(\nu ; J_1, \dots, J_\nu; \la_0, \eta_1, \dots , \eta_\nu)$ to just $\Cliff(\nu)$.
\end{definition}

\begin{remark}\label{rem:climpliesos}
As it follows from Definition~\ref{d:cl}, the operators $J_i$ are skew-symmetric, orthogonal and satisfy the equations
$\<J_iX, J_jX\> = \K_{ij}\|X\|^2$ and $J_iJ_j+J_jJ_i =-2 \delta_{ij} \id$, for all $i,j =1, \dots, \nu$, and all
$X \in \Rn$. This implies that every algebraic curvature tensor with a Clifford structure is Osserman, as
by \eqref{eq:cs} the Jacobi operator has the form
$\mathcal{R}_XY =  \la_0 (\| X\|^2 Y - \< Y, X \> X) + \sum\nolimits_{i=1}^\nu 3 \eta_i \< J_iX, Y \> J_iX$, so
for a unit vector $X$, the eigenvalues of $\mathcal{R}_X$ are $\la_0$ (of multiplicity $n-1-\nu$ provided $\nu<n-1$),
$0$ and $\la_0+3\eta_i, \; i=1, \dots, \nu$.

From the fact that the $J_i$'s are anticommuting almost Hermitian structures it easily follows that the operators
$J_{i_1} \dots J_{i_m}$ with pairwise nonequal $i_j$'s, are skew-symmetric, if $m \equiv 1,2 \mod(4)$, and are
symmetric otherwise.
\end{remark}

It turns out that every Osserman algebraic curvature tensor has a Clifford structure in all the dimensions
except for $n=16$, and also in many cases when $n=16$, as follows from \cite{Nma} (Proposition~1 and the second last
paragraph of the proof of Theorem~1 and Theorem~2), \cite[Proposition~1]{Nmm} and \cite[Proposition~2.1]{Nbel}.
The only known counterexample is an algebraic curvature tensor with a Cayley structure:
$\mathcal{R}= a R^{\Oc}+b R^{\mathbb{S}}$, where $R^{\mathbb{S}}$ and $R^{\Oc}$ are the curvature tensors of the unit
sphere $\mathbb{S}^{16}(1)$ and of the Cayley projective plane $\Oc P^2$, respectively, and $a \ne 0$.

A Clifford structure $\Cliff(\nu)$ on the Euclidean space $\Rn$ turns it the into a Clifford module (we refer to
\cite[Part~1]{ABS}, \cite[Chapter~11]{Hus} for standard facts on Clifford algebras and
Clifford modules). Denote $\Cl(\nu)$ a \emph{Clifford algebra} on $\nu$ generators  $x_1, \dots, x_\nu$, an
associative unital algebra over $\br$ defined by the relations $x_ix_j+x_jx_i=-2 \K_{ij}$ (this condition determines
$\Cl(\nu)$ uniquely). The map $\rho: \Cl(\nu) \to \Rn$ defined on generators by $\rho(x_i) = J_i$ (and
$\rho(1) = \id$) is a representation of $\Cl(\nu)$ on $\Rn$. As all the $J_i$'s are orthogonal and skew-symmetric,
$\rho$ gives rise to an \emph{orthogonal multiplication} defined as follows. In the Euclidean space $\br^{\nu}$,
fix an orthonormal basis $e_1, \dots, e_\nu$. For every $u = \sum_{i=1}^{\nu} u_i e_i \in \br^{\nu}$ and every
$X \in \Rn$, define
\begin{equation}\label{eq:ortmult}
    J_uX = \sum\nolimits_{i=1}^{\nu} u_i J_iX
\end{equation}
(when $u=e_i$, we abbreviate $J_{e_i}$ to $J_i$). The map $J: \br^{\nu} \times \Rn \to \Rn$ defined by
\eqref{eq:ortmult} is an orthogonal multiplication: $\|J_uX\|^2=\|u\|^2\|X\|^2$ (similarly, we can define
an orthogonal multiplication $J: \br^{\nu+1} \times \Rn \to \Rn$ by
$J_uX = u_0 X + \sum\nolimits_{i=1}^{\nu} u_i J_iX$, for $u = \sum_{i=0}^{\nu} u_i e_i \in \br^{\nu+1}$, where
$e_0, e_1, \dots, e_\nu$ is an orthonormal basis for the Euclidean space $\br^{\nu+1}$).

For $X \in \Rn$, introduce the subspaces
\begin{equation*}
    \cJ X = \Span(J_1X, \dots, J_\nu X), \qquad \cI X = \Span(X, J_1X, \dots, J_\nu X).
\end{equation*}
Later we will also use the complexified versions of these subspaces, which we denote
$\JC X$ and ${\mathcal{I}_{\bc}} X$ respectively, for $X \in \bc^n$.

\subsection{Algebraic curvature tensors of dimension $\mathbf{16}$ with a Clifford structure}
\label{ss:cs16}

To find all the algebraic curvature tensors with a Clifford structure in dimension $16$, we need to find all
the possible ways of turning $\Rs$ into a $\Cl(\nu)$-module. A convenient way to describe them is by using the
octonions.

In general, the proof of Theorem~\ref{t:co} extensively uses computations in the octonion algebra $\Oc$
(in particular, the standard identities like $a^*=2 \<a,1\> 1 - a, \; \<a,b\>=\<a^*, b^*\>=\frac12(a^*b+b^*a), \;
a(ab)=a^2b, \; \<a, bc\> = \<b^*a, c\> = \<ac^*, b\>, \; (ab^*)c+(ac^*)b = 2\<b, c\>a, \;
\<ab,ac\>=\<ba,ca\>=\|a\|^2 \<b,c\>$, for any $a, b, c \in \Oc$, and similar ones, see e.g. \cite{BG}) and
the fact that $\Oc$ is a division algebra (in particular, any nonzero octonion is invertible:
$a^{-1} = \|a\|^{-2} a^*$). We will also use the \emph{bioctonions} $\Oc \otimes \bc$, the algebra over $\bc$ with
the same multiplication table as that for $\Oc$. As all the above identities are polynomial, they still hold for
bioctonions, with the complex inner product on $\bc^8$, the underlying linear space of $\Oc \otimes \bc$. However,
the bioctonion algebra is not a division algebra (and has zero-divisors: $(\mathrm{i}1+e_1)(\mathrm{i}1-e_1)=0$).

In the following lemma,  (which contains known facts, but will be convenient for us to refer to)
we call a representation $\rho$ of $\Cl(\nu)$ in $\Rn$
\emph{orthogonal}, if all the $\rho(x_i)$ are orthogonal. Representation $\rho_1, \rho_2$ are called equivalent
(respectively, orthogonally equivalent), if there exists $T \in \operatorname{GL}(n)$ (respectively,
$T \in \operatorname{O}(n)$)
such that $\rho_2(x)=T\rho_1(x)T^{-1}$, for all $x \in \Cl(\nu)$. For a representation $\rho$, the representation
$-\rho$ is defined on the generators of $\Cl(\nu)$ by $(-\rho)(x_i)=-\rho(x_i)$
(induced by the automorphism $\a:x_i \to -x_i$ of $\Cl(\nu)$).

\begin{lemma}\label{l:reps}
1. For any representation of a Clifford algebra $\Cl(\nu)$ in $\Rs, \nu \le 8$. There is exactly one, up to orthogonal
equivalence, orthogonal representations $\rho_8$ of $\Cl(8)$ in $\Rs$. It can be defined as follows.
Identify $\Rs$ and $\Ro$ with $\Oc \oplus^\perp \Oc$ and $\Oc$ respectively, via linear isometries. Then the
orthogonal multiplication \eqref{eq:ortmult} defined by $\rho_8$ is given by
    \begin{equation} \label{eq:Jp}
    J_p (a , b) = (b p, - a p^*), \quad \text{for} \quad p \in \Ro=\Oc, \; X=(a , b) \in \Rs=\Oc \oplus^\perp \Oc.
    \end{equation}

2. Any orthogonal representation of a Clifford algebra $\Cl(\nu), \; \nu \ge 4$, in $\Rs$ is either a restriction of
the representation $\rho_8$ to $\Cl(\nu) \subset \Cl(8)$, or, up to orthogonal equivalence, is $\pm \rho_7$, where
$\rho_7$ is  the following reducible orthogonal representation of $\Cl(7)$. Identify $\Rs$ and $\mathbb{R}^7$ with
$\Oc \oplus^\perp \Oc$ and $\Oc'=1^\perp$ respectively, via linear isometries. Then the orthogonal multiplication
\eqref{eq:ortmult} defined by $\rho_7$ is given by
\begin{equation}\label{eq:Jp7}
    J_p(a , b)=(ap, bp), \quad \text{for} \quad p \in \mathbb{R}^7=\Oc', \; X=(a , b) \in \Rs=\Oc \oplus^\perp \Oc.
\end{equation}
\end{lemma}
\begin{proof}
It is easy to see that if two orthogonal representations are equivalent, then they are orthogonally equivalent.
Indeed, suppose that $\rho_1(x_i)=J_i$ and $\rho_2(x_i)=TJ_iT^{-1}$, where $T \in \operatorname{GL}(n)$. Then, as both
$\rho_1$ and $\rho_2$ are orthogonal, we get $(T^tT)J_i=J_i(T^tT)$. As every $J_i$ commutes with $T^tT$, it also
commutes with  $S=\sqrt{T^tT}$, the unique symmetric positive definite matrix such that $S^2=T^tT$. But then
$TS^{-1} \in \operatorname{O}(n)$ and $\rho_2(x_i)=(TS^{-1})J_i(TS^{-1})^{-1}$.

By \cite[Table~6.5]{Hus}, $N_5=N_6=N_8= \mathbb{Z}, \; N_7=\mathbb{Z} \oplus \mathbb{Z}$, where $N_\nu$ is the free
abelian group generated by irreducible representations of $\Cl(\nu)$.

1. The fact that $\nu \le 8$ follows from \cite[Theorem~11.8.2]{Hus}. The orthogonal multiplication defined by
\eqref{eq:Jp} is indeed a representation of $\Cl(8)$ in $\Rs$ (which follows from the octonion identity
$(ap^*)q+(aq^*)p=2\<p,q\>a$). As $N_8=\mathbb{Z}$, a representation of $\Cl(8)$ in $\Rs$ is unique, up to orthogonal
equivalence, hence any orthogonal representation is orthogonally equivalent to the one defined by \eqref{eq:Jp}.

2. The restriction of $\rho_8$ to any $\Cl(\nu) \subset \Cl(8)$ defines an orthogonal representation of $\Cl(\nu)$
in $\Rs$. As $N_5=N_6=\mathbb{Z}$, any orthogonal representation of $\Cl(\nu), \; \nu=5, 6$, in $\Rs$ is equivalent
(hence orthogonally equivalent) to it.

For the algebra $\Cl(7)$, the group $N_7=\mathbb{Z}\oplus\mathbb{Z}$ is generated by two inequivalent representations
in $\Ro$. These are $\pm \sigma$, where on generators, $\sigma(x_i)$ is the right multiplication in $\Oc$ by the
imaginary octonion $x_i$ (note that $\prod_{i=1}^7\sigma(x_i)=\pm \id$, with $\pm$ replaced by $\mp$ for
$-\sigma$). Then there are exactly three (orthogonally) inequivalent (orthogonal) representations of $\Cl(7)$ in
$\Rs$: $\pm 2\sigma$ and $(-\sigma)\oplus\sigma$. As it follows from \eqref{eq:Jp7}, $\rho_7=2\sigma$. Moreover,
neither of $\pm \rho_7$ can be a restriction of $\rho_8$ to $\Cl(7)$, as $\prod_{i=1}^7\rho_8(x_i)=\pm \id$,
so $\prod_{i=1}^8\rho_7(x_i)=\pm\rho_7(x_8)$, which is skew-symmetric, thus contradicting Remark~\ref{rem:climpliesos}.
\end{proof}

Note that an algebraic curvature tensor with a Clifford structure does not change, if we change the signs of
(some or all of) the $J_i$'s, so it does not matter, which of $\pm \rho_7$ is defined by \eqref{eq:Jp7}.

In the proof of Theorem~\ref{t:co} we will use the following Lemma.

\begin{lemma}\label{l:octlemma}

1. Suppose that a Clifford structure on $\Rs$ is given by \eqref{eq:Jp}.
Let $N: \Rs \to \Rs$ be a quadratic form such that for all $i = 1, \ldots 8$, the cubic polynomial $\<N(X), J_iX\>$
is divisible by $\|X\|^2$. Then there exist a linear operator $A: \Rs \to \Ro$ and vectors $V, U \in \Rs$ such that
$N(X) = J_{A(X)}X + \<V, X\> X - U \|X\|^2$.

2. Suppose that a Clifford structure on $\Rs$ is given by \eqref{eq:Jp7}.
Let $N=(N_1,N_2): \Rs \to \Rs$ be a quadratic form, $u$ be a unit imaginary octonion, and $p \in \Rs$ be such that for
all $X=(a,b) \in \Rs$, $a^* N_1(X) + b^* N_2(X)=\<p,X\> \<a^*b,u\>$. Then there exists $m \in \Rs$ such that
$N(X)=\|X\|^2(m- \pi_{\cI X}m)$ $=\|X\|^2 m - \<m,X\>X - \sum_{i=1}^7\<J_iX,m\>J_iX$, where $\pi_{\cI X}m$ is the
orthogonal projection to $\cI X$.
\end{lemma}

\begin{proof}
1. For $X = (a, b) \in \Rs$, let $N(X) = (N_1(X), N_2(X))$, where $N_1, N_2: \Rs \to \Ro$ are quadratic forms.
By the assumption and  \eqref{eq:Jp}, for any $q \in \Ro$, the cubic polynomial
$\<N(X), J_qX\> = \<N_1,bq\>- \<N_2, aq^*\>$ is divisible by $\|X\|^2$, hence so is the polynomial vector
$b^* N_1 - N_2^* a$. Then there exists a linear operator $L: \Rs \to \Ro$ such that $b^* N_1(X) - N_2(X)^* a =
(\|a\|^2 +\|b\|^2) L(X)$, which simplifies to
\begin{equation}\label{eq:fgl}
    b^*(N_1(X) - b L(X)) = (L(X) a^* + N_2(X)^*) a.
\end{equation}
Let for $X = (a, b)$,
\begin{equation*}
 N_1(X) - b L(X) = \xi_1(b) +  P(a,b) + \phi (a), \quad L(X) a^* + N_2(X)^* = \xi_2(a) +  Q(a,b) + \psi(b),
\end{equation*}
where $\phi,\xi_1,\xi_2,\psi: \Ro \to \Ro$ are quadratic forms and $P, Q: \Ro \times \Ro \to \Ro$ are bilinear forms.
Collecting the terms of the same degree in $a$ and $b$ in \eqref{eq:fgl} we get $\xi_1 = \xi_2 = 0$ and
\begin{equation}\label{eq:system}
    b^* P(a,b) = \psi(b) a, \quad b^* \phi(a) = Q(a,b) a.
\end{equation}
Taking $b=1$ in the second equation we find $\phi(a) = Q(a, 1) a$ and so
$Q(a, b) a = b^* (Q(a, 1) a) = 2\<b, Q(a,1)\> a - Q(a, 1)^* (b a)$, which implies
$(a^*b^*) Q(a,1) = a^*(2\<b, Q(a,1)\>1 - Q(a, b)^*)$.

For every $b$ from the unit sphere $\mathbb{S}^6 \subset \Oc'$, we can now apply \cite[Lemma~6]{Nmm}, with
$Y = a^*, \, e = b^*$, $L(Y) = Q(a,1)$ and $F(Y) = 2\<b, Q(a,1)\>1 - Q(a, b)^*$. Then
$Q(a,1) = \<c(b), a^*\>1 + \<d(b), a^*\> b^* + a p(b)$, for some maps $c,d,p:\mathbb{S}^6 \to \Oc$.
As the left-hand side does not depend on $b$, we get $a(p(b_1)-p(b_2)) \in \Span(1, b_1, b_2)$, for all $a \in \Oc$
and all $b_1, b_2 \in \mathbb{S}^6$, so $p(b) = p$, a constant. Taking the real and the imaginary parts of the both
sides we obtain that $c(b) = c$ is also a constant and $d(b) = 0$. Then $Q(a,1) = \<c^*, a \>1 + a p$, and therefore
$\phi(a) = Q(a, 1) a = \<m, a\> a - \|a\|^2 p^*$, where $m = c^* + 2 p^*$. Then from \eqref{eq:system},
$Q(a,b) a = b^* \phi(a) = (\<m, a\> b^* - (b^*p^*)a^*)a$, so $Q(a,b) = \<m, a\> b^* - (b^*p^*)a^*$.
Similarly, from the first equation of \eqref{eq:system} we get
$\psi(b) = \<r, b^*\> b^* - \|b\|^2 q, \; P(a,b) = \<r, b^*\> a - b(qa)$. Then
\begin{align*}
N_1(X) &= b L(X) +\<r, b^*\> a - b(q a) + \<m, a\> a -\|a\|^2 p^*, \\
N_2(X) &= -a L(X)^* + \<m, a\> b - a(p b) + \<r, b^*\> b -\|b\|^2 q^*,
\end{align*}
which implies $N(X) = J_{A(X)}X + \<V, X\> X - U \|X\|^2$, for all $X \in \Rs$, where $A: \Rs \to \Ro$ is a linear
operator defined by $A(X) = L(X) + b^*p^* - qa$, and $V = (m, r^*), \; U = (p^*, q^*) \in \Rs$.

2. From the assumption, $N_1((a,0))=N_2((0,b))=0$, so
$N_1((a,b))= P(a,b) + \xi_1(b), \; N_2((a,b))= \xi_2(a) + Q(a,b)$, for some quadratic forms $\xi_1, \xi_2:
\Ro \to \Ro$ and bilinear forms $P, Q: \Ro \times \Ro \to \Ro$. Collecting the terms of degree two in $a$ we get
$a^* P(a,b) + b^* \xi_2(a)=\<p_1,a\> \<a^*b,u\>$, where $p_1 \in \Ro, \; p=(p_1, p_2)$. Substituting $b=1$ we get
$\xi_2(a)=\<p_1,a\> \<a^*,u\>-a^* P(a,1)$, so
$a^* P(a,b) -b^*(a^* P(a,1))=\<p_1,a\> (\<a^*b,u\>- \<a^*,u\>b^*)$. Multiplying by $a$ from the left and taking
$a \perp p_1$ we get $\|a\|^2 P(a,b) - a(b^*(a^* P(a,1)))=0$ which can be rewritten as
$ -2(\<a,b\>-\<a,1\>b)(a^* P(a,1)) = \|a\|^2(bP(a,1) - P(a,b))$. Multiplying by $(\<a,b\>-\<a,1\>b)^*$ from the
left we obtain that all the components of the polynomial vector $\|\<a,b\>-\<a,1\>b\|^2(a^* P(a,1))$ belong to the
ideal $\mathbf{I}$ generated by $\|a\|^2$ and $\<p_1,a\>$. For a fixed nonzero $b \perp 1$, the quadratic form
$\|\<a,b\>-\<a,1\>b\|^2$ is not in $\mathbf{I}$ (as it is nonzero and vanishes on $\Span(1, b)^\perp)$. As
the ideal $\mathbf{I}$ is prime, it follows that all the components of $a^* P(a,1)$ belong to $\mathbf{I}$.
Since $P(a,1)$ is linear in $a$, we obtain $a^* P(a,1)=\|a\|^2 c + \<p_1,a\>La$, for some $c_1 \in \Oc$ and some
linear operator $L$ on $\Oc$. Multiplying this by $a$ from the left we obtain that $a \cdot La$ is divisible by
$\|a\|^2$, so $a \cdot La=\|a\|^2c_2$, for some $c_2 \in \Oc$, so $La=a^*c_2$, therefore
$a^* P(a,1)=\|a\|^2 c_1 + \<p_1,a\>a^*c_2$ which implies $P(a,1)=a c_1 + \<p_1,a\>c_2$. Then
$a^* P(a,b) -b^*(a^* (a c_1 + \<p_1,a\>c_2))=\<p_1,a\> (\<a^*b,u\>- \<a^*,u\>b^*)$, for all $a, b \in \Oc$,
so $a^* P(a,b) - \|a\|^2 b^*c_1 =\<p_1,a\> (\<a^*b,u\>- \<a^*,u\>b^*+b^*(a^* c_2))$.

Assume $p_1 \ne 0$. Then multiplying by $a$ from the left we obtain that all the components of the polynomial vector
$a(\<a^*b,u\>- \<a^*,u\>b^*+b^*(a^* c_2))$ are divisible by $\|a\|^2$. As it is linear in $b$ and quadratic in $a$, we
obtain $a(\<a^*b,u\>- \<a^*,u\>b^*+b^*(a^* c_2))=\|a\|^2 L_1b$ for some linear operator $L_1$ on $\Oc$, so
$\<a^*b,u\>- \<a^*,u\>b^*+b^*(a^* c_2)=a^* \cdot L_1b$. Taking $a=1$ we get $L_1b=\<b,u\>+b^*c_2$, so
$\<a^*b,u\>- \<a^*,u\>b^*- \<b,u\>a^*=a^* (b^*c_2)-b^*(a^* c_2)$, for all $a, b \in \Oc$. This implies that for all
orthogonal $a, b \in \Oc', \; a(bc_2) \in \Span(a,b,1)$, so $bc_2 \in \Span(a,ab,1)$, so $c_2 \in \Span(ab,a,b)$.
Thus $c_2=0$, so $\<a^*b,u\>- \<a^*,u\>b^*- \<b,u\>a^*=0$, for all $a, b \in \Oc$, which leads to a contradiction
(take $b=u, \; a \perp 1, u$).

It follows that $p_1=0$, so $a^* P(a,b) - \|a\|^2 b^*c_1 =0$ which implies $P(a,b) =a(b^*c_1)$, and then
$\xi_2(a)=\<p_1,a\> \<a^*,u\>-a^* P(a,1)=-\|a\|^2 c_1$. Similarly, $Q(a,b) =b(a^*c_3)$ and $\xi_1(b)=-\|b\|^2 c_3$,
for some $c_3 \in \Oc$. It follows that $N(X)=(a(b^*c_1) -\|b\|^2 c_3,b(a^*c_3)-\|a\|^2 c_1)$. Then by \eqref{eq:Jp7},
$N(X)=\|X\|^2(m- \pi_{\cI X}m)=\|X\|^2 m - \<m,X\>X - \sum_{i=1}^7\<J_iX,m\>J_iX$, where $m=(-c_3,-c_1) \in \Rs$.
\end{proof}

\subsection{Algebraic curvature tensors with a Cayley structure}
\label{ss:actcayley}

The curvature tensor $R^{\Oc}$ of the Cayley projective plane $\Oc P^2$ of the sectional
curvature between $1$ and $4$, is explicitly given in \cite[Eq.~6.12]{BG}. Identifying
the tangent space $T_x\Oc P^2$ with  $\Oc \oplus \Oc$ via a linear isometry, we have for
$X=(x_1, x_2), \, Y=(y_1, y_2), Z=(z_1, z_2) \in T_x\Oc P^2 = \Oc \oplus \Oc$:
\begin{align*}
R^{\Oc}(X,Y)Z=
(&4\<x_1,z_1\>y_1-4\<y_1,z_1\>x_1-(z_1y_2)x_2^*+(z_1x_2)y_2^*-(x_1y_2-y_1x_2)z_2^*,\\
&4\<x_2,z_2\>y_2-4\<y_2,z_2\>x_2-x_1^*(y_1z_2)+y_1^*(x_1z_2)+z_1^*(x_1y_2-y_1x_2)).
\end{align*}
Introducing the symmetric operators $S_i \in \End(\Rs), \; i=0,1, \dots, 8$, by
$S_0 (x_1, x_2)= (x_1, -x_2)$, $S_i (x_1, x_2)= (e_ix_2^*, x_1^*e_i), \; i=1, \dots, 8$, where $\{e_i\}$ is an
orthonormal basis for $\Oc$, we obtain
\begin{equation}\label{eq:Rop2}
    R^{\Oc}(X,Y)= 3 X \wedge Y + \sum\nolimits_{i=0}^8 S_i X \wedge S_i Y,
\end{equation}
where $X \wedge Y$ is the skew-symmetric operator defined by $(X \wedge Y)Z=\<X,Z\>Y-\<Y,Z\>X$.

As it follows from the definition, the operators $S_i$ are orthogonal and satisfy
\begin{equation}\label{eq:SiSj}
S_iS_j+S_jS_i=2 \K_{ij} \id, \quad 0 \le i, j \le 8.
\end{equation}
For every $w$ in the Euclidean space $\Rd$, introduce the symmetric operator
$S_w=\sum_{i=0}^8 w_i S_i$. As it follows from \eqref{eq:SiSj}, the map $S: \Rd \times \Rs \to \Rs$ defined by
$(w, X) \to S_wX$ is an orthogonal multiplication: $\|S_wX\|=\|w\| \cdot \|X\|$, for all $w \in \Rd, \; X \in \Rs$.
We usually abbreviate $S_{e_i}$ to $S_i$.

The operators $S_i$ define the structure of the Clifford $\Cl^+(9)$-module on the Euclidean space $\Rs$ as follows.
Denote $\Cl^+(9)$ the Clifford algebra on nine generators  $x_0, x_1, \dots, x_8$, an associative unital algebra over
$\br$ defined by the relations $x_ix_j+x_jx_i=2 \K_{ij}$. The map $\sigma: \Cl^+(9) \to \Rs$ defined on generators by
$\sigma(x_i)=S_i$ (and $\sigma(1) = \id$) is a representation of $\Cl^+(9)$ on $\Rs$. The Clifford algebra $\Cl^+(9)$
is isomorphic to $\br(16) \oplus \br(16)$, where $\br(16)$ is the algebra of $16 \times 16$ real matrices
\cite[\S4]{ABS}, so $\sigma$ is surjective. In particular, as by \eqref{eq:SiSj} the operator $\prod_{i=0}^8 S_i$
commutes with all the $S_i$'s (hence with all the $\br(16)$) and is orthogonal, we have $\prod_{i=0}^8 S_i=\pm \id$.

As by \eqref{eq:SiSj}, for every nonzero $w \in \Rd, \; S_w^2=\|w\|^2\id$ and $\Tr S_w=0$ (multiply \eqref{eq:SiSj} by
$S_i$ and take the trace), $S_w$ has two eigenvalues $\pm\|w\|$, each of multiplicity $8$. Denote $E_{\|w\|}(S_w)$ the
$\|w\|$-eigenspace of $S_w$ and $\pi_{E_{\|w\|}(S_w)}$ the orthogonal projection of $\Rs$ to $E_{\|w\|}(S_w)$. Then
$\pi_{E_{\|w\|}(S_w)}=\frac12(\|w\|^{-1}S_w+\id)$.

Introduce the subspaces $\mL_k=\Span_{i_1 < \dots < i_k}(S_{i_1} \dots S_{i_k}) \subset \End(\Rs), \; 0 \le k \le 9$
(in particular, $\mL_0= \br \, \id$). On $\End(\Rs)$, introduce the inner product by $\<Q_1, Q_2\>=\Tr (Q_1 Q_2^t)$
and the operator $\mA$ by
\begin{equation}\label{eq:mAdef}
\mA Q=\sum\nolimits_{i=0}^8 S_i Q S_i, \qquad Q \in \End(\Rs).
\end{equation}


\begin{lemma}\label{l:mA}
1. For $k=0, \dots, 9$, $\mL_{9-k}=\mL_k$. Moreover, $\oplus_{k=0}^4 \mL_k = \End(\Rs)$ and
$\Sym(\Rs)=\mL_0 \oplus \mL_1 \oplus \mL_4, \; \Sk(\Rs)=\mL_2 \oplus \mL_3$, where $\Sk(\Rs)$ and $\Sym(\Rs)$
are the spaces of the skew-symmetric and the symmetric endomorphisms of $\Rs$ respectively, and all the direct sums
are orthogonal. Moreover, $\{\frac{1}{4}S_{i_1} \dots S_{i_k}, \; i_1 < \dots < i_k\}$ is an orthonormal basis for
$\mL_k, \; 0 \le k \le 4$.

2. The operator $\mA$ is symmetric and does not depend on the choice of an orthonormal basis $\frac14 S_i$ for
$\mL_1$. Its eigenspaces are $\mL_k, \; 0 \le k \le 4$, with the corresponding eigenvalues $(-1)^k(9-2k)$. In
particular, $\Sk(\Rs)$ and $\Sym(\Rs)$ are invariant subspaces of $\mA$.

3. For every $w \in \Rd, \; Q \in \End(\Rs), \;
\mA(QS_w)=-\mA(Q)S_w+2S_wQ, \; \mA(S_wQ)=-S_w\mA(Q)+2QS_w$.
For every $X, Y \in \Rs, \quad \mA(X \wedge Y)=\sum\nolimits_{i=0}^8 S_i X \wedge S_i Y$.

4. For every $X \in \Rs, \; X \in \Span_{i=0}^8(S_iX)$ and $\sum_{i=0}^8\<S_iX,X\>S_iX=\|X\|^2X$.
For every $X, Y \in \Rs$, $\sum_{i=0}^8(2\<S_iX,Y\>S_iX+\<S_iX,X\>S_iY)=\|X\|^2Y+2\<X,Y\>X$.
For a unit vector $X \in \Rs, \; R^{\Oc}_XY=Y$, when $Y \perp \Span_{i=0}^8(S_iX)$ and $R^{\Oc}_XY=\frac14 Y$,
when $Y \in \Span_{i=0}^8(S_iX), \, Y \perp X$.


5. Let $N: \Rs \times \Rs \to \Rs$ be a bilinear skew-symmetric map such that $\<N(X,Y),Z\>=0$, for every
$w \in \Rd$ and for every $X, Y, Z \in E_{\|w\|}(S_w)$. Then there exists $q \in \Rs$ such that
\begin{equation}\label{eq:n111=0}
    N(X,Y)=(\mA-\id)(X \wedge Y)q=\sum\nolimits_{i=0}^8 (\<S_i X,q\>S_i Y-\<S_i Y,q\>S_i X) - (\<X,q\> Y-\<Y,q\> X).
\end{equation}
\end{lemma}
\begin{proof}
1. The fact that $\mL_{9-k}=\mL_k$ follows from $\prod_{i=0}^8 S_i=\pm \id$. Moreover, all the operators
$\frac{1}{4}S_{i_1} \dots S_{i_k}$, $i_1 < \dots < i_k, \; 0 \le k \le 4$, are orthonormal.
Indeed, from \eqref{eq:SiSj}, the norm of each of them is $1$. The inner product of two different ones is
$\frac{1}{16}$ times the trace of some $S'= S_{j_1} \dots S_{j_p}, \; j_1 < \dots < j_p, \; 1 \le p \le 8$, which is
clearly zero, if $S'$ is skew-symmetric. Otherwise, $p=1$ or $p=4$ (by $\mL_{9-k}=\mL_k$), and in the both cases,
$\Tr S' = 0$, as $S'$ is a product of a symmetric and a skew-symmetric operator: $S_i=S_i \cdot (S_iS_j)$,
when $p=1$, and $S_iS_jS_kS_l=S_i \cdot (S_jS_kS_l)$, when $p=4$ (for arbitrary pairwise nonequal $i,j,k,l$).
Now, as the $S_i$'s are symmetric, \eqref{eq:SiSj} implies that $\mL_0 \oplus \mL_1 \oplus \mL_4  \subset \Sym(\Rs)$
and that $\mL_2 \oplus \mL_3 \subset \Sk(\Rs)$, with both inclusions being in fact equalities by the dimension count.

2. 3. Directly follow from (\ref{eq:SiSj}, \ref{eq:mAdef}) and the fact that $S_i(X \wedge Y)S_i=S_iX \wedge S_iY$.

4. For $X=(x_1,x_2) \ne 0, \; x_1, x_2 \in \Oc$, define $w \in \Rd=\br \oplus \Oc$ by
$w=\|X\|^{-2}((\|x_1\|^2-\|x_2\|^2)e_0$ $+2x_1x_2)$. Then $S_wX=X$ by the definition of the $S_i$'s,
so $X \in \Span_{i=0}^8(S_iX)$. As $\<S_iX,S_jX\>=\K_{ij}\|X\|^2$ by \eqref{eq:SiSj}, equation
$\sum_{i=0}^8\<S_iX,X\>S_iX=\|X\|^2X$ follows. The second equation is obtained by polarization. Substituting it to
the expression for $R^{\Oc}_XY$ obtained from \eqref{eq:Rop2}, we get the last statement of the assertion.

5. For every $Z \in \Rs$, define $K(Z) \in \Sk(\Rs)$ by $\<K(Z)X,Y\>=\<N(X,Y),Z\>$.
As $\pi_{E_{\|w\|}(S_w)}=\frac12(\|w\|^{-1}S_w+\id)$ we have
$(\|w\|^{-1}S_w+\id)K((\|w\|^{-1}S_w+\id)Z)(\|w\|^{-1}S_w+\id)=0$, for all nonzero $w \in \Rd$, so
$(S_wK(S_wZ)S_w+\|w\|^2(K(S_wZ)+S_wK(Z)+K(Z)S_w))+\|w\|(S_wK(S_wZ)+K(S_wZ)S_w+S_wK(Z)S_w+\|w\|^2K(Z))=0$.
As $\|w\|$ is not a rational functions, we obtain
\begin{equation}\label{eq:NSSS}
S_wK(S_wZ)S_w+\|w\|^2(K(S_wZ)+S_wK(Z)+K(Z)S_w)=0,
\end{equation}
for all $w \in \Rd$ and all $Z \in \Rs$.
It follows that for every fixed $Z \in \Rs$ and for every $i,j=0, \dots, 8$, the expression
$\Tr(S_wK(S_wZ)S_w S_iS_j)$, viewed as a polynomial of $w = (w_0, \dots, w_8)$, is divisible by $\|w\|^2$
(the $w_i$'s are the coordinates of $w$ relative to the orthonormal basis $\{e_0,e_1, \dots , e_8\}$ for $\Rd$ such
that $S_{e_i}=S_i$). Then from \eqref{eq:SiSj} we obtain that
$\|w\|^2 \mid w_i \Tr(K(S_wZ)S_jS_w)-w_j \Tr(K(S_wZ)S_iS_w)$. For every $Z \in \Rs$, define the
polynomials $F_{Z,i}(w)=\Tr(K(S_wZ)S_iS_w)$. Then $\|w\|^2 \mid w_i F_{Z,j}(w)-w_j F_{Z,i}(w)$.
Let $\mathbf{I}$ be the ideal of the polynomial ring
$\mathbf{K}=\br[w_0,w_1, \dots, w_8]$ generated by $\|w\|^2$ and let $\pi: \mathbf{K} \to \mathbf{K}/\mathbf{I}$ be
the natural projection. Then $\pi(w_i F_{Z,j}-w_j F_{Z,i})=0$, so the $2 \times 9$-matrix over the ring
$\mathbf{K}/\mathbf{I}$ whose $i$-th row is $(\pi(w_i),  \pi(F_{Z,i})), \; i=0,1,\dots,8$, has rank at most one.
As the ring $\mathbf{K}/\mathbf{I}$ is a unique factorization domain \cite{Nag}, there exist
$x,y,u_i \in \mathbf{K}/\mathbf{I}$, such that $\pi(w_i)=xu_i, \, \pi(F_{Z,i})=yu_i$. Since the
elements $\pi(w_i)=xu_i$ are coprime, $x$ is invertible, so we can take $x=1$, hence $\pi(F_{Z,i})=y\pi(w_i)$.
Lifting this equation up to $\mathbf{K}$ we obtain that for every $Z \in \Rs$, there exist polynomials $Y_Z(w)$ and
$g_{Z,i}(w)$ such that $F_{Z,i}(w)=Y_Z(w)w_i +\|w\|^2g_{Z,i}(w)$. As $F_{Z,i}$ is a quadratic form,
the comparison of terms of the same degree gives that for every $Z$, we can choose $Y_Z(w)$ to be a linear form in $w$
and the $g_{Z,i}(w)$'s to be constants. Using the linearity by $Z$ be obtain that for some linear maps $T,G:\Rs\to\Rd$,
$\Tr(K(S_wZ)S_iS_w)=F_{Z,i}(w)=\<TZ,w\>w_i +\|w\|^2\<GZ,e_i\>$, for all $Z \in \Rs, \; w \in \Rd, \; i=0,1, \dots, 8$.
Substituting $S_wZ$ as $Z$ we obtain $\|w\|^2\Tr(K(Z)S_iS_w)=\<TS_wZ,w\>w_i +\|w\|^2\<GS_wZ,e_i\>$. It follows that
$\|w\|^2 \mid \<TS_wZ,w\>=\<Z,S_wT^tw\>$, for all $Z \in \Rs$, so $S_wT^tw=\|w\|^2p$, for some $p \in \Rs$. As
$S_w^2=\|w\|^2\id$, this implies $T^tw=S_wp$, so $\Tr(K(Z)S_iS_w)=\<Z,p\>w_i +\<GS_wZ,e_i\>$. Multiplying both sides by
$w_i$ and summing up by $i=0,1, \dots, 8$ we get $0=\<Z,p\>\|w\|^2 +\<GS_wZ,w\>$ (the left-hand side vanishes, as
$S_w^2=\|w\|^2\id$ and $K(Z)$ is skew-symmetric). It follows that $\|w\|^2 p +S_wG^tw=0$, so $G^tw=-S_w p$. Hence
$\Tr(K(Z)S_iS_w)=\<Z,p\>w_i +\<S_wZ,G^te_i\>=\<Z,p\>w_i -\<S_wZ,S_ip\>$. In particular, for $w=e_j,\; j \ne i$, we
get by \eqref{eq:SiSj}: $\Tr(K(Z)S_iS_j)=\<Z,S_iS_jp\>=-\frac12\<Z \wedge p,S_iS_j\>$. As
$\{\frac14S_iS_j, \; i <j\}$ is an orthonormal basis for $\mL_2$ by assertion~1, we obtain
$\pi_2 K(Z)=-\frac12 \pi_2(Z\wedge p)$, for all $Z \in \Rs$, where $\pi_2$ is the orthogonal projection from
$\Sk(\Rs)$ to $\mL_2$.

Now for any $K \in \Sk(\Rs), \; \pi_2(K)= \frac18(\mA+3\id)K$ and $\pi_3(K)= -\frac18(\mA-5\id)K$,
(where $\pi_i$ is the orthogonal projection to $\mL_i$) by assertions~1 and 2, and
$\mA(S_wKS_w)=S_w\mA(K)S_w$ and $\mA(S_wK+KS_w)=-S_w\mA(K)-\mA(K)S_w+2S_wK+2KS_w$, by assertion~3, so
$\pi_2(S_wK+KS_w)=S_w\pi_3(K)+\pi_3(K)S_w$ and $\pi_2(S_wKS_w)=S_w\pi_2(K)S_w$. Projecting \eqref{eq:NSSS} to $\mL_2$
we then obtain $S_w\pi_2(K(S_wZ))S_w+\|w\|^2(\pi_2(K(S_wZ))+S_w\pi_3(K(Z))+\pi_3(K(Z))S_w)=0$. As it is shown above,
$\pi_2 K(Z)=-\frac12 \pi_2(Z\wedge p)$, so
$-\frac12 S_w\pi_2((S_wZ)\wedge p)S_w+\|w\|^2(-\frac12 \pi_2((S_wZ)\wedge p)+S_w\pi_3(K(Z))+\pi_3(K(Z))S_w)=0$,
which (using the fact that $S_w\pi_2(K)S_w=\pi_2(S_wKS_w)$ and that $S_w(X \wedge Y) S_w= (S_wX) \wedge (S_wY)$)
simplifies to $S_w\pi_3(K(Z))+\pi_3(K(Z))S_w=\frac12 \pi_2(Z\wedge (S_wp)+(S_wZ)\wedge p)$. As
$S_w(X\wedge Y)+(X\wedge Y) S_w= (S_wX) \wedge Y +X \wedge (S_wY)$ and $\pi_2(S_wK+KS_w)=S_w\pi_3(K)+\pi_3(K)S_w$,
we obtain $S_w\pi_3(K(Z)-\frac12 Z\wedge p)+\pi_3(K(Z)-\frac12 Z\wedge p)S_w=0$. Taking $w=e_i$ we get
$S_i\pi_3(K(Z)-\frac12 Z\wedge p)S_i=-\pi_3(K(Z)-\frac12 Z\wedge p)$, which implies
$\mA(\pi_3(K(Z)-\frac12 Z\wedge p))$ $=-9\pi_3(K(Z)-\frac12 Z\wedge p)$, so
$\pi_3(K(Z)-\frac12 Z\wedge p)=0$ by assertion~2.

As $K(Z)=\pi_2(K(Z))+\pi_3(K(Z))$ by assertion~1, it follows that
$K(Z)=\frac12(-\pi_2+\pi_3)$ $(Z\wedge p)=\frac{1}{16}( -(\mA+3\id)-(\mA-5\id))(Z\wedge p)
=-\frac{1}{8}(\mA-\id)(Z\wedge p)$. Since $\<K(Z)X,Y\>=-\frac12\<K(Z),X \wedge Y\>$ and $\mA$ is symmetric,
$\<N(X,Y),Z\>=\<K(Z)X,Y\>=\frac{1}{16}\<(\mA-\id)(Z\wedge p),X \wedge Y\>
=\frac{1}{16}\<(\mA-\id)(X \wedge Y),Z\wedge p\>=\frac{1}{8}\<(\mA-\id)(X \wedge Y)p,Z\>$, so
$N(X,Y)=(\mA-\id)(X \wedge Y)q$, with $q=\frac{1}{8}p$, as required.
\end{proof}

\section{Conformally Osserman manifolds. Proof of Theorem~\ref{t:co} and Theorem~\ref{t:ws}}
\label{s:conf}

In all the cases, except when $n=16$, Theorem~\ref{t:co} is already proved: for the Osserman Conjecture, see
\cite[Theorem~2]{Nma}, for the Conformal Osserman Conjecture, see \cite[Theorem~1]{Nco}. In this section, we will
first prove Theorem~\ref{t:co} for conformally Osserman manifold of dimension $16$ (assuming Conjecture~A) and then
deduce from it the proof for ``genuine" Osserman manifolds.

Theorem~\ref{t:ws} is an easy corollary of Theorem~\ref{t:co}, as in Theorem~\ref{t:ws} we consider the Riemannian
manifolds $M^n$, for which Conjecture~A is already ``satisfied" by the assumption: the Weyl tensor of each of them at
every point is proportional either to the Weyl tensor of $\bc P^n$ or $\mathbb{H} P^n$ (hence has a Clifford
structure) or to the Weyl tensor of $\Oc P^2$. By Theorem~\ref{t:co}, $M^n$ is locally conformally equivalent to a
rank-one symmetric space, which is, in fact, $M_0^n$, as the Weyl tensors of different rank-one symmetric spaces are
different (for instance, because their Jacobi operators have different multiplicities of the eigenvalues).

We start with a brief informal outline of the proof of the conformal part of Theorem~\ref{t:co}. Recall that the Weyl
tensor of a Riemannian manifold $M^n$ is defined by
\begin{equation}\label{eq:defWeyl}
    R(X,Y)= \hat \rho X \wedge Y - \hat \rho Y \wedge X + W(X,Y),
\end{equation}
where $\hat \rho=\frac{1}{n-2}\operatorname{Ric}-\frac{\operatorname{scal}}{2(n-1)(n-2)}\id$,
$\operatorname{Ric}$ is the Ricci operator, $\operatorname{scal}$ is the scalar curvature and $X \wedge Y$ is
the skew-symmetric operator defined by $(X \wedge Y)Z=\<X,Z\>Y-\<Y,Z\>X$.
According to Conjecture~A, the Weyl tensor has either a Clifford or a Cayley structure. First of all, in
Section~\ref{ss:smooth} we show that both
these structures can be chosen smooth on an open, dense subset $M' \subset M^{16}$ (see Lemma~\ref{l:locc1} for the
precise statement), so that on every connected component $M_\a$ of $M'$, the curvature tensor $R$ of $M^{16}$ is given
by either \eqref{eq:confcs} or \eqref{eq:op2ctc} (up to a conformal equivalence), with all the operators and the
functions involved being locally smooth.
Then we establish the local version of the theorem, at every point $x \in M'$, for each of the two cases separately.
This is done by using the differential Bianchi identity and the fact that under a conformal change of the
metric, the symmetric tensor field $\rho$ (which is a linear combination of $\hat\rho$ from \eqref{eq:defWeyl} and
the identity) is a Codazzi  tensor, that is, $(\n_X\rho) Y=(\n_Y\rho) X$. Using the result of \cite{DS}, we show
that $\rho$ must be a constant multiple of the identity, which implies that every connected component
$M_\a \subset M'$ is locally conformally equivalent to a symmetric Osserman manifold.
The proof in the Clifford case is given in Section~\ref{ss:cc} (Lemma~\ref{l:nablarho} and
Lemma~\ref{l:codazzi}), in the Cayley case, in Section~\ref{ss:oc} (Lemma~\ref{l:op2bi}).
Then, by the result of \cite[Lemma 2.3]{GSV}, every $M_\a$ is either locally conformally flat or is locally conformal
to a rank-one symmetric space.
In Section~\ref{ss:glo} we prove the conformal part of Theorem~\ref{t:co} globally, by
first showing (using Lemma~\ref{l:mmeps}) that $M$ splits into a disjoint union of a closed subset $M_0$, on which the
Weyl tensor vanishes, and nonempty open connected subsets $M_\a$, each of which is locally conformal to one of the
rank-one symmetric spaces. On every $M_\a$, the conformal factor $f$ is a well-defined positive smooth function.
Assuming that there exists at least one $M_\a$ and that $M_0 \ne \varnothing$ we show that there exists a point
$x_0 \in M_0$ on the boundary of a geodesic ball $B \subset M_\a$ such that both $f(x)$ and $\n f(x)$ tend to zero when
$x \to x_0, \; x \in B$ (Lemma~\ref{l:x_0}). Then the positive function $u=f^{7/2}$ satisfies an elliptic equation
in $B$, with $\lim_{x \to x_0, x \in B}u(x)=0$, hence by the boundary point theorem, the limiting value of the inner
derivative of $u$ at $x_0$ must be positive. This contradiction implies that either $M=M_0$ or $M=M_\a$, thus
proving the conformal part of Theorem~\ref{t:co}. The ``genuine Osserman" part of Theorem~\ref{t:co} then follows
easily using the result of \cite{Nic}.

\subsection{Smoothness of the Clifford and of the Cayley structures}
\label{ss:smooth}

Let $M^{16}$ be a connected smooth Riemannian manifold whose Weyl tensor at every point is Osserman.
Define a function $N: M^{16} \to \mathbb{N}$ as follows: for $x \in M^{16}, \; N(x)$ is the number of distinct
eigenvalues of the operator $W_{X|X^\perp}$, where $W_X$ is the Jacobi operator associated to the Weyl tensor and
$X$ is an arbitrary nonzero vector from $T_xM^{16}$. As the Weyl tensor is Osserman, the function $N(x)$ is
well-defined. Moreover, as the set of symmetric operators having no more than $N_0$ distinct eigenvalues is closed in
the linear space of symmetric operators on $\br^{15}$, the function $N(x)$ is lower semi-continuous (every subset
$\{x \, : \, N(x) \le N_0\}$ is closed in $M^{16}$). Let $M'$ be the set of points where the function $N(x)$ is
continuous (that is, locally constant). It is easy to see that $M'$ is an open and dense (but possibly disconnected)
subset of $M^{16}$. The following lemma shows that, assuming Conjecture~A, all the ``ingredients" of the curvature
tensor are locally smooth on every connected component of $M'$.

\begin{lemma} \label{l:locc1}
Let $M^{16}$ be a smooth conformally Osserman Riemannian manifold. Let $M'$ be the (open, dense) subset of all the
points of $M^{16}$ at which the number of distinct eigenvalues of the Jacobi operator associated to the Weyl tensor of
$M^{16}$ is locally constant.

Assume Conjecture A.
Then for every $x \in M'$, there exists a neighborhood $\mU=\mU(x)$ with exactly one of the following properties.
\begin{enumerate}[(a)]
    \item \label{it:cl}
    There exists $\nu \ge 0$, smooth functions $\eta_1, \dots, \eta_{\nu}: \mU \to \br \setminus \{0\}$, a smooth
    symmetric linear operator field $\rho$ and smooth anticommuting almost Hermitian structures
    $J_i, \; i=1, \dots, \nu$, on $\mU$ such that for all $y \in \mathcal{U}$ and all $X, Y, Z \in T_yM^{16}$,
    the curvature tensor of $M^{16}$ has the form
\begin{multline}\label{eq:confcs}
R(X, Y) Z = \< X, Z \> \rho Y +\<\rho X, Z\> Y - \< Y, Z \> \rho X -\<\rho Y, Z\> X \\
+ \sum\nolimits_{i=1}^\nu \eta_i (2 \< J_iX, Y \> J_iZ + \< J_iZ, Y \> J_iX - \< J_iZ, X \> J_iY).
\end{multline}

    \item \label{it:oc}
The Riemannian manifold $\mU$ is conformally equivalent to a Riemannian manifold whose curvature tensor has the form
\begin{equation}\label{eq:op2ctc}
    R(X,Y)=\rho X \wedge Y - \rho Y \wedge X + \ve \sum\nolimits_{i=0}^8 S_i X \wedge S_i Y
    = \rho X \wedge Y - \rho Y \wedge X + \ve \mA(X \wedge Y),
\end{equation}
at every point $y \in \mU$, where $\ve = \pm 1$ and $\rho, S_i, \; i=0, \dots, 8$, are smooth fields of symmetric
operators on $\mU$ satisfying \eqref{eq:SiSj}.
\end{enumerate}
\end{lemma}

\begin{proof}
On every connected component $M_\a \subset M'$, the number $N=N(x)$ is a constant, so the operator $W_{X|X^\perp}$,
where $X$ is a unit vector, has exactly $N$ distinct eigenvalues $\mu_0, \mu_1, \dots, \mu_{N-1}$, with the
multiplicities $m_0, m_1, \dots, m_{N-1}$. The functions $\mu_i$'s are smooth on $M_\a$, and the $m_i$'s are
constants, by the smoothness of the characteristic polynomial of $W_{X|X^\perp}$. We label them in such a way that
$m_0=\max(m_0,m_1,\dots, m_{N-1})$.

Clearly, $\sum_{i=0}^{N-1} m_i=15$ and, as $\Tr W_X=0$, we have $\sum_{i=0}^{N-1} m_i\mu_i=0$. It follows that if
$N=1$, then $W=0$, so $M_\a$ is conformally flat. Then by \eqref{eq:defWeyl}, the curvature tensor has the form
\eqref{eq:confcs}, with $\nu=0$ and a smooth $\rho$.

Suppose $N > 2$. By Conjecture A, $W$ either has a Clifford structure, or a Cayley structure. But in the latter case,
the operator $W_{X|X^\perp}$ has two distinct eigenvalues (from assertion~4 of Lemma~\ref{l:mA}).
It follows that $W$ has a Clifford structure $\Cliff (\nu)$, at every point of $M_\a$ ($\nu$ may a priori depend on
$x \in M_\a$). By assertion~1 of Lemma~\ref{l:reps}, $\nu \le 8$, and by Remark~\ref{rem:climpliesos},
for a unit vector $X$, the eigenvalues of $W_{X|X^\perp}$ are $\la_0$, of multiplicity $15-\nu$, and
$\la_0+3\eta_i, \; i=1, \dots, \nu$. All of the $\eta_i$'s are nonzero (by Definition~\ref{d:cl}), some of them can
be equal, but not all, as otherwise $N=2$, so the multiplicity of every eigenvalue $\la_0+3\eta_i$ is at most
$\nu-1 \le 15 -\nu$, as $\nu \le 8$. It follows that the maximal multiplicity is $m_0 = 15 - \nu \, (\ge 7)$, so
$\nu = 15-m_0$, which is a constant on $M_\a$. Moreover, $\la_0=\mu_0$ (this is automatically satisfied, unless
$\nu=8$ and $\eta_1= \dots =\eta_7 \ne \eta_8$; in the latter case we have two eigenvalues of multiplicity $7$ and
we choose the labeling of the $\mu_i$'s so that $\mu_0=\la_0$). The functions $\la_0$ and $\la_0+3\eta_i$ are smooth,
as each of them equals one of the $\mu_i$'s. Moreover, for every smooth unit vector field $X$ on $M_\a$ and every
$i=1, \dots, N-1$, the $\mu_i$-eigendistribution of $W_{X|X^\perp}$ (which must be smooth on $M_\a$ and must have a
constant dimension $m_i$) is $\Span_{j:\la_0+3\eta_j=\mu_i}(J_jX)$, by Remark~\ref{rem:climpliesos}.
By assertion~3 of Lemma~3 of \cite{Nco}, there exists a neighborhood $\mU_i(x)$ and smooth anticommuting almost
Hermitian structures $J_j'$ (for the $j$'s such that $\la_0+3\eta_j=\mu_i$) on $\mU_i(x)$ such that
$\Span_{j:\la_0+3\eta_j=\mu_i}(J_jX)=\Span_{j:\la_0+3\eta_j=\mu_i}(J_j'X)$. Let $W'$ be a (unique)
algebraic curvature tensor on $\mU= \cap_{i=1}^{N-1}\mU_i(x)$ with the Clifford structure
$\Cliff(\nu; J'_1,\dots, J'_\nu; \la_0, \eta_1, \dots , \eta_\nu)$. Then $\nu=15-m_0$ is constant and all
the $J'_i, \eta_i$ and $\la_0$ are smooth on $\mU$. Moreover, for every unit vector field $X$ on
$\mU$, the Jacobi operators $W'_X$ and $W_X$ have the same eigenvalues and eigenvectors by construction, hence
$W'_X=W_X$, which implies $W'=W$. Then the curvature tensor on $\mU$ has
the form \eqref{eq:confcs}, with the operator $\rho$ given by
$\rho=\frac{1}{n-2}\operatorname{Ric}+(\frac12 \la_0-\frac{\operatorname{scal}}{2(n-1)(n-2)})\id$, by
\eqref{eq:defWeyl}. As $\la_0$ is a smooth function, the operator field $\rho$ is also smooth.

Now consider the case $N=2$. Again, by Conjecture A, $W$ either has a Clifford structure, or a Cayley structure.
In the former case, by
Remark~\ref{rem:climpliesos}, for a unit vector $X$ at every point $x \in M_\a$, the eigenvalues of $W_{X|X^\perp}$
are $\la_0$, of multiplicity $15-\nu$, and $\la_0+3\eta, \;\eta \ne 0$, of multiplicity $\nu$. In the latter case,
there are two eigenvalues, of multiplicities $m_0=8$ and $m_1=7$, respectively (as it follows from assertion~4 of
Lemma~\ref{l:mA}).
In the both cases, $m_0 \ge 8$. It follows that if $m_0 >8$, the Weyl tensor $W$ has a Clifford structure
$\Cliff(m_1)$, at every point $x \in M_\a$. Then we can finish the proof as in the case $N>2$ considered above,
as on $M_\a$, the functions $\la_0=\mu_0$ and $\la_0+3\eta=\mu_1$ are smooth and $\Span_{j=1}^{m_1}(J_jX)$, the
$\mu_1$-eigendistribution of $W_{X|X^\perp}$, is smooth and has a constant dimension $m_1$. Assertion~3 of
Lemma~3 of \cite{Nco} applies and we obtain the curvature tensor of the form \eqref{eq:confcs} (with $\nu=m_1$ and
all the $\eta_i$'s equal) on some neighborhood $\mU=\mU(x)$ of an arbitrary point $x \in M_\a$.

Suppose now that $N=2$, and $m_0=8, \, m_1=7$. Then at every point $x \in M_\a$, the Weyl tensor either has a Clifford
structure $\Cliff(7)$, with $\eta_1=\dots=\eta_7=\eta$, or a Clifford structure $\Cliff(8)$, with
$\eta_1=\dots=\eta_8=\eta$, or a Cayley structure. Denote $M^{(7)},M^{(8)}$ and $M^{(\Oc)}$
the corresponding subsets of $M_\a$, respectively. These three subsets are mutually disjoint. Indeed, if
$W=a R^{\Oc}+b R^{\mathbb{S}}, \; a \ne 0$, would have a Clifford structure, then the same would be true for
$R^{\Oc}=a^{-1}W-ba^{-1}R^{\mathbb{S}}$ (as by Definition~\ref{d:cl}, the set of algebraic curvature tensors with a
Clifford structure is invariant under scaling and shifting by a constant curvature tensor). This contradicts the fact
that $R^{\Oc}$ has no Clifford structure \cite[Remark~1]{Nma}. Moreover, if $x \in M^{(7)} \cap M^{(8)}$, then for any
unit vector $X \in T_xM^{16}$ the operator $W_{X|X^\perp}$ has an eigenspace of dimension $7$ spanned by orthonormal
vectors $J_1X, \dots, J_7X$ (where the $J_i$ are defined by the Clifford structure $\Cliff(7)$), and the orthogonal
eigenspace of dimension $8$ spanned by orthonormal vectors $J_1'X, \dots, J_8'X$ (where the $J_j'$ are defined by
$\Cliff(8)$). Then the fifteen operators $J_i, J_j'$ are anticommuting almost Hermitian structures on $\Rs$, which
contradicts the fact that $\nu \le 8$ (assertion~1 of Lemma~\eqref{l:reps}). It follows that
$M_\a=M^{(7)}\sqcup M^{(8)}\sqcup M^{(\Oc)}$. Moreover, each of the three subsets $M^{(7)}, M^{(8)}, M^{(\Oc)}$ is
relatively closed in $M_\a$. Indeed, suppose a sequence $\{x_n\} \subset M^{(7)}$ converges to a point $x^0 \in M_\a$.
Then the functions $\la(x_n)$ and $\la(x_n)+\eta(x_n)$ are well-defined (as $M^{(7)},M^{(8)},M^{(\Oc)}$ are disjoint)
and bounded (as $\la^2(x_n)+(\la(x_n)+\eta(x_n))^2=\mu_0(x_n)^2+\mu_1(x_n)^2$), so we can assume that both sequences
converge by choosing a subsequence; moreover, as all the $J_i(x_n)$ are orthogonal operators, we can choose a
subsequence such that all of them converge. Then by continuity, $W(x^0)$ has the form \eqref{eq:cs}, with $\nu=7$,
with the $J_i$'s being anticommuting almost Hermitian structures, and with $\eta \ne 0$, as otherwise $N=1$, which
contradicts the fact that $N=2$ on $M_\a$. Therefore, $x^0 \in M^{(7)}$. The above arguments work for $M^{(8)}$
almost verbatim (by replacing $7$ by $8$), and for $M^{(\Oc)}$, with a slight modification (by replacing the
orthogonal operators $J_i$'s by the orthogonal operators $S_i$ from \eqref{eq:Rop2}). Hence, as $M_\a$ is connected,
it coincides with exactly one of the sets $M^{(7)},M^{(8)},M^{(\Oc)}$.

Now, if $M_\a=M^{(7)}$ or if $M_\a=M^{(8)}$, the proof follows from the same arguments as in the case $N>2$: we have
a Clifford structure for $W$ with a constant $\nu$.

Suppose $M_\a=M^{(\Oc)}$. Then by (\ref{eq:Rop2}, \ref{eq:defWeyl}), the Weyl tensor on $M_\a$ has the form
$W(X,Y)=b X \wedge Y + a \sum\nolimits_{i=0}^8 S_i X \wedge S_i Y$, with $a \ne 0$ (actually, $5b=3a \ne 0$, as
$\Tr W_X=0$; see \eqref{eq:model2}), so by \eqref{eq:defWeyl}, the curvature tensor at every point $x \in M_\a$ has
the form
\begin{equation}\label{eq:op2ct}
    R(X,Y)=\rho X \wedge Y - \rho Y \wedge X + f \sum\nolimits_{i=0}^8 S_i X \wedge S_i Y,
\end{equation}
where $\rho$ is a symmetric operator and $f \ne 0$ (as $N=2$ on $M_\a$). As $S_i^2=\id$ (by \eqref{eq:SiSj}) and
$\Tr S_i=0$ (see the proof of assertion~1 of Lemma~\ref{l:mA}), at every point $x \in M_\a$, the Ricci operator,
the scalar curvature and the Weyl tensor of $M^{16}$ are given by
$\operatorname{Ric} X=14\rho X +(\Tr\rho - 9f)X,  \; \operatorname{scal}=30\Tr\rho - 144f, \;
W(X,Y)=\frac{1}{5}f(3 X \wedge Y + 5 \sum\nolimits_{i=0}^8 S_i X \wedge S_i Y$, respectively.
A direct computation using \eqref{eq:SiSj} and the fact that the $S_i$'s are symmetric,
orthogonal and $\Tr S_i=0$ gives $\|W\|^2=\frac{32256}{5}f^2$ (see \eqref{eq:model3}). As $f \ne 0$,
it follows that $f$ is a smooth function (hence $\rho$ is a smooth symmetric operator, as $\operatorname{scal}$ and
$\operatorname{Ric}$ are smooth) on $M_\a$. Introduce an algebraic curvature tensor $P$ defined by
$P(X,Y)=f^{-1}(R(X,Y)-\rho X \wedge Y + \rho Y \wedge X)=\sum_{i=0}^8 S_i X \wedge S_i Y$ $=\mA(X \wedge Y)$,
where the last equation follows from assertion~3 of Lemma~\ref{l:mA}.
As $P$ is smooth, the field $\mA$ of the endomorphisms of the bundle $\Sk(M_\a)$ of skew-symmetric endomorphisms
over $M_\a$ is smooth (the fact that $\mA$ is an endomorphism of the bundle $\Sk(M_\a)$ follows from assertion~2 of
Lemma~\ref{l:mA}). Then the eigenbundles of $\mA, \; \mL_2(M_\a)$ and $\mL_3(M_\a)$ (assertion~2 of Lemma~\ref{l:mA})
are also smooth. As the matrix product is smooth, it follows that the subbundle
$(\mL_2(M_\a))^2=\Span(K_1K_2, \,K_1, K_2 \in \mL_2(M_\a))$ is smooth. By the definition of $\mL_2$ and from
\eqref{eq:SiSj}, $(\mL_2(M_\a))^2=\mL_0(M_\a) \oplus \mL_2(M_\a) \oplus \mL_4(M_\a)$, with the direct sum being
orthogonal relative to the (smooth) inner product in $\End(M_\a)$ (assertion~1 of Lemma~\ref{l:mA}). Since
$\mL_2(M_\a)$ is smooth and $\mL_0(M_\a)$ is a one-dimensional bundle spanned by the identity operator, the bundle
$\mL_4(M_\a)$ is smooth. Then the bundle
$(\mL_4(M_\a))^2= \bigoplus_{s=0}^4 \mL_{2s}(M_\a)=\bigoplus_{s=0}^4 \mL_{s}(M_\a)$ (the latter equation follows from
assertion~1 of Lemma~\ref{l:mA}) is also smooth. The direct sum on the right-hand side is orthogonal with respect to
the smooth inner product and the bundles $\mL_s(M_\a), \; s=0,2,3,4$, are smooth, as it is shown above.
Hence $\mL_1(M_\a)$
is a smooth subbundle of $\End(M_\a)$. It follows that on some neighborhood $\mU(x)$ of an arbitrary point $x\in M_\a$
we can choose nine smooth sections $S'_i, \; i=0,1, \dots, 8$, of $\mL_1(M_\a)$, which are orthogonal and all have
norm $4$. By assertion~2 of Lemma~\ref{l:mA}, the operator $\mA$ does not change, if we replace $S_i$ by $S'_i$,
so by assertion~3 of Lemma~\ref{l:mA}, the curvature tensor \eqref{eq:op2ct} remains unchanged if we replace
the operators $S_i$ by the smooth operators $S'_i$. Therefore we can assume $f, \rho$ and the $S_i$'s in
\eqref{eq:op2ct} to be smooth on $\mU$.
Under a conformal change of metric $\tilde g =h\<\cdot,\cdot\>$, for a positive smooth function
$h=e^{2\phi}: \mU \to \br$, the curvature tensor transforms as $\tilde R(X,Y)=R(X,Y)-(X \wedge KY + KX \wedge Y)$,
where $K=H(\phi)-\n\phi \otimes \n\phi + \tfrac12 \|\nabla \phi\|^2 \id$ and $H(\phi)$ is the symmetric
operator associated to the Hessian of $\phi$ (see Lemma~\ref{l:mmeps}). As $X \tilde\wedge Y=h X \wedge Y$ we obtain
$\tilde R(X,Y)=\tilde \rho X \tilde \wedge Y - \tilde \rho Y \tilde \wedge X +
fh^{-1} \sum_{i=0}^8 S_i X \tilde \wedge S_i Y$, where $\tilde \rho = h^{-1}(\rho-K)$. Taking $h=|f|$
(and dropping the tildes) we obtain \eqref{eq:op2ctc}, with $\ve=\pm 1 = \operatorname{sgn}(f)$.
\end{proof}

\subsection{Clifford case}
\label{ss:cc}

Let $x\in M'$ and let $\mathcal{U}=\mathcal{U}(x)$ be the neighborhood of $x$ defined in assertion~\eqref{it:cl} of
Lemma~\ref{l:locc1}.
By the second Bianchi identity, $(\n_UR)(X,Y) Y + (\n_YR)(U,X) Y+(\n_XR)(Y,U) Y = 0$. Substituting $R$ from
\eqref{eq:confcs} and using the fact that the operators $J_i$'s and their covariant derivatives are skew-symmetric and
the operator $\rho$ and its covariant derivatives are symmetric we get:
\begin{equation} \label{eq:confBZY}
\begin{split}
&\<X,Y\> ((\n_U \rho)Y - (\n_Y \rho)U)+\|Y\|^2 ((\n_X \rho)U-(\n_U \rho) X) +\<U,Y\> ((\n_Y \rho)X-(\n_X \rho) Y)\\
&+\<(\n_Y \rho)U-(\n_U \rho)Y, Y\> X +\<(\n_X \rho)Y-(\n_Y \rho)X, Y\> U +\<(\n_U \rho)X-(\n_X \rho)U, Y\> Y \\
&+ \sum\nolimits_{i=1}^\nu 3(X(\eta_i) \< J_iY, U \> - U(\eta_i) \< J_iY, X \>)J_iY  \\
&+ \sum\nolimits_{i=1}^\nu Y(\eta_i) (2 \< J_iU, X \> J_iY + \< J_iY, X \> J_iU - \< J_iY, U \> J_iX) \\
&+ \sum\nolimits_{i=1}^\nu \eta_i
\bigl((3 \< (\n_UJ_i)X, Y \> + 3 \< (\n_XJ_i)Y, U \> + 2 \< (\n_YJ_i)U, X \>)J_iY\\
&+3 \< J_iX, Y \> (\n_UJ_i) Y + 3 \< J_iY, U \> (\n_XJ_i) Y + 2 \< J_iU, X \> (\n_YJ_i) Y\\
&+ \< (\n_YJ_i) Y, X \> J_iU + \< J_i Y, X \> (\n_YJ_i) U
- \< (\n_YJ_i) Y, U \> J_iX - \< J_i Y, U \> (\n_YJ_i) X \bigr)=0.
\end{split}
\end{equation}
Taking the inner product of \eqref{eq:confBZY} with $X$ and assuming $X, Y$ and $U$ to be orthogonal we obtain
\begin{equation} \label{eq:confBZYX}
\begin{split}
&\|X\|^2 \<Q(Y),U\> +\|Y\|^2 \<Q(X),U\>-\<Q(X),Y\>\<Y,U\>-\<Q(Y),X\>\<X,U\>\\
+& \sum\nolimits_{i=1}^\nu 3(X(\eta_i) \< J_iY, U \> - Y(\eta_i) \<J_iX,U\> - U(\eta_i) \<J_iY,X\>)\<J_iY,X\> \\
+& \sum\nolimits_{i=1}^\nu 3 \eta_i
\bigl((2 \< (\n_UJ_i)X, Y \> + \< (\n_XJ_i)Y, U \> + \< (\n_YJ_i)U, X \>)\<J_iY,X\>\\
& \hphantom{\sum\nolimits_{i=1}^\nu \la_i (}
 - \< J_iY, U\> \<(\n_XJ_i)X, Y\> - \< J_iX, U\> \<(\n_YJ_i) Y,X\>\bigr)=0,
\end{split}
\end{equation}
where $Q: \Rs \to \Rs$ is the quadratic map defined by
\begin{equation}\label{eq:defQ}
\<Q(X), U\>=\<(\n_X \rho)U-(\n_U \rho) X,X\>.
\end{equation}
Note that $\<Q(X), X\> = 0$.

\begin{lemma}\label{l:nablarho}
In the assumptions of Lemma~\ref{l:locc1}, let $x \in M'$ and let $\mU$ be the neighborhood of $x$ introduced in
assertion~\eqref{it:cl} of Lemma~\ref{l:locc1}. Suppose that $\nu >4$. For every point $y \in \mU$, identify
$T_yM^{16}$ with the Euclidean space $\Rs$ via a linear isometry. Then
\begin{enumerate}[1.]
  \item
There exist $m_i, b_{ij} \in \Rs, \; i,j =1, \dots, \nu$, such that for all $X \in \Rs$ and all
$i,j =1, \dots, \nu$,
\begin{subequations}\label{eq:lnablarhoi}
\begin{gather} \label{eq:Qsum}
Q(X)=3 \sum\nolimits_{k=1}^\nu \<m_k,X\>J_kX, \\
\label{eq:nXJX}
    (\n_X J_i)X = \eta_i^{-1}(\|X\|^2 m_i-\<m_i,X\>X)+\sum\nolimits_{j=1}^\nu \<b_{ij},X\>J_jX,\\
\label{eq:bijbji}
    b_{ij}+b_{ji}=\eta_i^{-1}J_jm_i+\eta_j^{-1}J_im_j.
\end{gather}
\end{subequations}

  \item
  The following equations hold for all $X, Y \in \Rs$ and all $i,j =1, \dots, \nu$:
\begin{subequations}\label{eq:lnablarhoii}
\begin{gather}
\label{eq:neta}
    \nabla \eta_i = 2 J_im_i,\\
\label{eq:JiJkY}
\eta_i b_{ij} + \eta_j b_{ji}=0, \quad i \ne j, \\
\label{eq:skewrho}
(\n_Y \rho)X-(\n_X \rho)Y=\sum\nolimits_{i=1}^\nu (2\<J_iY,X\>m_i - \<m_i,Y\>J_iX+\<m_i,X\>J_iY),\\
\label{eq:bijne}
b_{ij} (\eta_i - \eta_j)=0, \\
\label{eq:allequal}
J_i m_i = \eta_i p, \quad \text{for some $p \in \Rs$}.
\end{gather}
\end{subequations}
\end{enumerate}
\end{lemma}
\begin{proof}

1. Equation \eqref{eq:confBZYX} is a polynomial equation in $48$ real variables, the coordinates of the vectors
$X, Y, U$. It must still hold if we allow $X, Y, U$ to be complex and extend the $J_i$'s, the $\n J_i$'s and
$\<\cdot , \cdot \>$ by complex linearity (bilinearity) to $\bc^{16}$. The complexified inner product
$\<\cdot,\cdot\>$ is a nonsingular quadratic form on $\bc^{16}$ (not a Hermitian inner
product on $\bc^{16}$).

From \eqref{eq:confBZYX}, for any two vectors $X, Y \in \bc^{16}$ with $Y \perp \IC X$, we get
\begin{multline} \label{eq:8c}
\|X\|^2 Q(Y) +\|Y\|^2 Q(X)-\<Q(X),Y\>Y-\<Q(Y),X\>X\\
-\sum\nolimits_{i=1}^\nu 3 \eta_i \bigl(\<(\n_XJ_i)X, Y\> J_iY + \<(\n_YJ_i) Y,X\> J_iX\bigr)=0.
\end{multline}

By assertion~2 of Lemma~\ref{l:reps}, the Clifford structure has one of two possible forms.
We prove identities \eqref{eq:lnablarhoi} separately for each of them.

\smallskip

\textbf{Case (a)} The representation of $\Cl(\nu)$ is a restriction to $\Cl(\nu) \subset \Cl(8)$ of the
representation $\rho_8$ of $\Cl(8)$ given by \eqref{eq:Jp}.

By complexification, we can assume that $J_pX$ is given by equation \eqref{eq:Jp}, where $X=(a,b) \in \bc^{16}$ and
$a, b, p \in \OC$. Denote $\cp = \{X=(a, b)\, : \, \|X\|^2=0\} \subset \bc^{16}$ the isotropic cone, and
for $X \in \bc^{16}$, denote $\IC^9 X$ the complex linear span of $X, J_1X, \dots, J_8X$. For $X \in \cp$, the space
$\IC^9 X$ is isotropic: the inner product of any two vectors from $\IC^9 X$ vanishes. Take $Y = J_q X \in \IC^9 X$,
with some $q \in \OC$. Then from \eqref{eq:8c} we obtain:
$\sum\nolimits_{i=1}^\nu \eta_i \<(\n_XJ_i)X, J_q X\> J_iJ_q X \in \IC^9 X$. Introduce the ($\bc$-)linear operator
$M_X: \OC \to \OC$ by $M_X(q) = \sum_{i=1}^\nu \eta_i \< (\n_XJ_i)X, J_q X \> e_i$. Then, as $\IC^9 X$ is isotropic,
we get $J_{M_X(q)}J_qX \perp \IC^9 X$, for all $X \in \cp$ and all $q \in \OC$. Then by \eqref{eq:Jp}, for any
$X = (a, b) \in \cp$ and any $p \in \OC$ we obtain $((- a q^*) M_X(q), - (b q)(M_X(q))^* \perp (b p, - a p^*)$, so
\begin{equation}\label{eq:mxq1}
    (M_X(q)(b q)^*) a = b^*((a q^*) M_X(q)).
\end{equation}
The bioctonion equation $(m (b q)^*) a - b^*((a q^*) m)=0$ for $m \in \OC$, with $(a, b, q)$ on the algebraic surface
$\mS = \cp \times \bc^8 \subset \bc^{24}$, can be viewed as a system of eight linear equations for
$m \in \bc^8$. Let $\mathcal{M}(a,b,q)$ be the matrix of this system. As both $m = q$ and $m = b^*a$ are solutions to
the system, $\rk \mathcal{M}(a,b,q) \le 6$, for all the points $(a, b, q)$ from a nonempty Zariski open subset of
$\mS$. On the other hand, if $a=q=1$, $b \perp 1, \, \|b\|^2 = -1$, the equation has the form $m b^* = b^* m$,
which implies that $m \in \Span_{\bc}(1, b)$, so $\rk \mathcal{M}(a,b,q) \ge 6$ for all the points $(a,b,q)$ from a
nonempty Zariski open subset of $\mS$. It follows that for a nonempty Zariski open subset of $\mS$,
$\rk \mathcal{M}(a,b,q) = 6$ and the solution set is $\Span_{\bc}(q, b^*a)$. Therefore from \eqref{eq:mxq1},
$M_X(q) \in \Span_{\bc}(q, b^*a)$, for all $(a,b,q)$ from a nonempty Zariski open subset of the $\mS$, hence
$\rk (q, M_X q, b^*a) \le 3$, for all $X = (a, b) \in \cp, \; q \in \bc^8$. It follows that the linear operators from
$\bc^8$ to $\Sk(\bc^8)$ defined by $q \to M_Xq \wedge (b^*a)$ and $q \to q \wedge (b^*a)$ are linearly dependent,
for every $X = (a, b) \in \cp$, so $M_Xq \wedge (b^*a) = c_X q \wedge (b^*a)$, for $c_X \in \bc$. Then
$(M_Xq-c_X q)\wedge (b^*a) = 0$, so $M_X(q)=c_X q + \a_X(q) b^*a$, for all $X=(a, b) \in \cp$ such that $b^*a \ne 0$,
where $\a_X$ is a linear form on $\bc^8$. By the definition of
$M_X(q), \; c_X q + \a_X(q)b^*a = \sum_{i=1}^\nu \eta_i \< (\n_XJ_i)X, J_q X \> e_i$.
Substituting $q=e_j, \; j \le \nu$, and taking the inner product with $e_k, \; k \le \nu$, we obtain
$\eta_k^{-1} (c_X \K_{jk} + \a_X(e_j)\<b^*a,e_k\>) = \<(\n_XJ_k)X, J_j X \>$. As the right-hand side is antisymmetric
in $j$ and $k$, we get
$2\eta_k^{-1} c_X \K_{jk} + \a_X(e_j) \cdot \eta_k^{-1}\<b^*a,e_k\>+ \a_X(e_k) \cdot \eta_j^{-1}\<b^*a,e_j\> = 0$. As
the rank of the $\nu \times \nu$-matrix $A_X$ defined by
$(A_X)_{jk}=\a_X(e_j) \cdot \eta_k^{-1}\<b^*a,e_k\> + \a_X(e_k) \cdot \eta_j^{-1}\<b^*a,e_j\>$
is at most two and as $\nu > 4$, we obtain that $c_X=0$, for all $X=(a, b) \in \cp$ such that $b^*a \ne 0$, hence
$A_X=0$. Now, if $\nu=8$, it follows from $A_X=0, \; b^*a \ne 0$, that $\a_X=0$, so $M_X = 0$. If $\nu < 8$, then, as
$c_X=0$, we have $M_X(q)=\a_X(q)b^*a = \sum_{i=1}^\nu \eta_i \< (\n_XJ_i)X, J_q X \> e_i$, so $\a_X(q)\<b^*a,e_s\>=0$,
for all $s >\nu$. Choosing $X=(a, b) \in \cp$ such that all the components of $b^*a$ are nonzero, we again obtain that
$\a_X=0$, so $M_X = 0$. It follows that $M_X = 0$ for all $X$ from a nonempty open subset of $\cp$, hence for all
$X \in \cp$.

From the definition of $M_X(q)$, it follows that for all $X \in \cp$ and all
$i=1,\dots,\nu, \; j=1, \dots, 8$, $\<(\n_X J_i)X, J_jX\> = 0$, so the polynomials $\<(\n_X J_i)X, J_jX\>$ are
divisible by $\|X\|^2$ over $\bc$, and hence over $\br$.
Then by assertion~1 of Lemma~\ref{l:octlemma}, $(\n_X J_i)X =\sum_{j=1}^8 \<b_{ij},X\>J_jX+ \<V_i, X\>X-\|X\|^2 U_i$,
for all $i=1,\dots,\nu$, where $b_{ij}, V_i, U_i$ are some vectors from $\Rs$ (note that the summation on the
right-hand side is up to $8$, not up to $\nu$, as in \eqref{eq:nXJX}). As $\<(\n_X J_i)X,X\>=0$, we have $V_i=U_i$,
so
\begin{equation}\label{eq:nJwith8}
(\n_X J_i)X =\sum\nolimits_{j=1}^8 \<b_{ij},X\>J_jX+ \eta_i^{-1}(\|X\|^2 m_i-\<m_i,X\>X),
\end{equation}
for some $m_i \in \Rs$. Substituting this into \eqref{eq:8c}, complexifying the resulting equation and taking
$X \in \cp$, $Y=J_qX$, with some $q \in \OC$ we get $\<Q(X),J_qX\>J_qX-\<Q(J_qX),X\>X=0$. As $X$ and $J_qX$ are
linearly independent for a nonempty Zariski open subset of $(X,q) \in \mS=\cp \times \bc^8$, the polynomial
$\<Q(X),J_qX\>$ vanishes on $\mS$, so for all $q \in \OC$, the polynomial $\<Q(X),J_qX\>$ is divisible by
$\|X\|^2$. Then $\<Q(X),J_iX\>$ is divisible by $\|X\|^2$, for all $i=1, \dots, 8$, which by assertion~1 of
Lemma~\ref{l:octlemma}, implies that $\<Q(X),Y\>=\|X\|^2\<Y,U\>$, for some fixed $U \in \Rs$, where $Y \perp \cI^9 X$.

It then follows from \eqref{eq:8c} and \eqref{eq:nJwith8} that for all $X, Y \in \Rs$, with $Y \perp \cI^9 X$,
\begin{equation} \label{eq:8cq}
\|X\|^2 T(Y)+\|Y\|^2 T(X)=0,
\end{equation}
where the quadratic form $T: \Rs \to \Rs$ is defined by $T(X)=Q(X)-\<X,U\>X -3 \sum_{i=1}^\nu \<m_i,X\> J_iX$.
We want to show that $T=0$. Suppose that for some $E \in \Rs$, the
quadratic form $t(X)=\<T(X),E\>$ is nonzero. Then from \eqref{eq:8cq}, $\|X\|^2 t(Y)+\|Y\|^2 t(X)=0$, for all
$X, Y \in \Rs,\; Y \perp \cI^9 X$. If $X=(a,0)$, $Y=(b,0)$, with $a \perp b$, then $Y \perp \cI^9 X$ by \eqref{eq:Jp},
so $\|a\|^2 t((b,0))+\|b\|^2 t((a,0))=0$ which implies $t((a,0))=0$, for all $a \in \Oc$. Similarly,
$t((0,b))=0$, for all $b \in \Oc$. It follows that $t((a,b))=\<La,b\>$ for some
$L \in \End(\Oc)$. From \eqref{eq:Jp}, any $X=(a,b), \; a, b \ne 0$, and any
$Y=(\|b\|^{-2}bu, \|a\|^{-2}au^*)$, with $u \perp b^*a$, satisfy $Y \perp \cI^9 X$. Then from
$\|X\|^2 t(Y)+\|Y\|^2 t(X)=0$ and $t((a,b))=\<La,b\>$ we obtain $\|u\|^2 \<La,b\>=\<L(bu),au^*\>$, for all
$u \perp b^*a$ (the condition $a,b \ne 0$ can be dropped). It follows that $\<\|u\|^2L^tb-L(bu)\cdot u,a\>=0$ for
all $a \perp bu$ (where $L^t$ is the operator transposed to $L$), so $\|u\|^2L^tb-L(bu)\cdot u \parallel bu$, for all
$b, u \in \Oc$. Taking the inner product with $bu$ we get $\|u\|^2L^tb-L(bu)\cdot u =0$. Taking $u=1$ we get $L^t=L$,
so $Lb \cdot u^* = L(bu)$, for all $b, u \in \Oc$. Substituting $b=1$ we obtain $Lu=pu^*$, where $p=L1$, so
$(pb^*)u^* = p(bu)^*$. If $p \ne 0$, this equation, with $b=p$, implies $p^* u^* = u^*p^*$, for all $u \in \Oc$. This
contradiction shows that $p=0$, hence $L=0$, hence $t=0$.

Therefore, the quadratic form $T(X)=Q(X)-\<X,U\>X -3 \sum_{i=1}^\nu \<m_i,X\> J_iX$ vanishes, which implies
$Q(X)=\<X,U\>X +3 \sum_{i=1}^\nu \<m_i,X\> J_iX$. Substituting this and \eqref{eq:nJwith8} into \eqref{eq:8c}, with
$Y \perp \cI^9 X$, we get $U=0$, which proves \eqref{eq:Qsum}.

Now, if $\nu=8$, then \eqref{eq:nXJX} follows from \eqref{eq:nJwith8}. If $\nu < 8$, choose
$X \ne 0, \; Y=J_sX, \; s > \nu$. Substituting into \eqref{eq:8c} and using
\eqref{eq:nJwith8} and \eqref{eq:Qsum} we get $\sum_{i=1}^\nu \eta_i(\<b_{is},X\>J_iJ_sX+\<b_{is},J_sX\>J_iX)=0$.
Taking $X=(a,0)$ and $X=(0,a)$ and using \eqref{eq:Jp} we get $\<b_{is},(a,0)\>=\<b_{is},(0,a)\>=0$, so $b_{is}=0$,
for all $1 \le i \le \nu < s \le 8$. This, together with \eqref{eq:nJwith8}, proves \eqref{eq:nXJX}.

Equation \eqref{eq:bijbji} follows from \eqref{eq:nXJX} and the fact that $\<(\n_XJ_i)X, J_jX\>$ is antisymmetric in
$i$ and $j$.

\smallskip

\textbf{Case (b)} $\nu=7$ and the representation of $\Cl(7)$ is given by \eqref{eq:Jp7}.

Let $e_i, \; i=1, \dots ,7$, be a fixed orthonormal basis in $\Oc'=1^\perp$ (or in its complexification), for
instance, the one with the multiplication table as in \cite[Section~3.64]{Bes}.

As in Case (a), by complexification, we can assume that $J_pX$ is given by equation \eqref{eq:Jp7}, where
$X=(a,b) \in \bc^{16}$ and $a, b, p \in \OC, \; p \perp 1$. We extend $J_pX$ to $\OC$ by complex linearity by defining
$J_{a \cdot 1}X=aX$, for $a \in \bc$. Denote $\cp$ the isotropic cone in $\bc^{16}$, and for $X \in \bc^{16}$, denote
$\IC X$ the complexification of $\cI X$. Take $X \in \cp, \; q \in \OC$. Then $Y = J_q X \in \IC X$, so by
\eqref{eq:8c}, $\sum_{i=1}^7 \eta_i \<(\n_XJ_i)X, J_q X\> J_iJ_q X \in \IC X$. Introduce the operator
$M_X \in \End(\OC)$ by
\begin{equation}\label{eq:defMXq7}
M_X(q) = \sum\nolimits_{i=1}^7 \eta_i \< (\n_XJ_i)X, J_q X \> e_i.
\end{equation}
As $\IC X$ is isotropic, $J_{M_X(q)}J_qX \perp J_p X$, for all $q,p \in \OC$. Then by \eqref{eq:Jp7},
for any $X \in \cp, \; q \in \OC$,
\begin{equation*}
    (a q)^*(a M_X(q))+ (b q)^*(b M_X(q))=0.
\end{equation*}
Consider this bioctonion equation as a linear system for $M \in \End(\OC)$,
with $X =(a, b)\in \cp$. A direct computation shows that $M(q)=q,a^*(bq)$ and $M(q)=\<v,q\>1, \<w,q\>a^*b$, with
arbitrary $v, w \in \OC$, are the solutions. When $a=e_2, \; b= \mathrm{i} e_3$,
these solutions span a subspace of dimension $18$ of $\End(\OC)$, so for all $X=(a,b)$ from a nonempty Zariski open
subset $\cp_1 \subset \cp$, the dimension of the solution space is at least $18$. On the other hand,
a direct computation, with $a=e_2, \; b= \mathrm{i} e_3$ shows that 
every solution is a linear combination of those above. It follows that the corank of the matrix of the linear system
(whose entries are polynomials in the coordinates of $X=(a,b) \in \cp$) equals $18$ for all the points $X=(a,b)$ from
a nonempty Zariski open subset $\cp_2 \subset \cp$. Then for every $X=(a,b) \in \cp_3= \cp_1 \cap \cp_2$, the operator
$M$ is be a linear combination of the four listed above, that is,
\begin{equation*}
    M_X(q)= \<v_X,q\>1 + \<w_X,q\>a^*b + \a_X q + \b_X a^*(bq),
\end{equation*}
for all $X=(a,b) \in \cp_3$, where $v, w: \cp_3 \to \OC, \; \a, \b: \cp_3 \to \bc$.
From \eqref{eq:defMXq7}, $M_X(1)=0, \; M_X(q) \perp 1$, so
\begin{equation}\label{eq:MXq}
    M_X(q)= \Im(\<w_X,q\>a^*b - \<w_X,1\> a^*(bq) + \a_X q),
\end{equation}
where $\Im$ is the operator of taking the imaginary part of a bioctonion: $\Im(q)=q-\<q,1\>1$.  Define the symmetric
operator $D$ on $\OC$ by $D1=0, \; De_i = \eta_i^{-1} e_i$. From \eqref{eq:defMXq7}
it follows that $\<M_Xq, Dq\>=0$, for all $q \in \OC$, so for all $X=(a,b) \in \cp_3$,
\begin{equation}\label{eq:diag}
    \<w_X,q\>\<a^*b, Dq\> = \<w_X,1\> \<a^*(bq), Dq\> - \a_X \<q, Dq\>.
\end{equation}
Substituting $q=b^*a$ and using the fact that $Dq \perp 1$ we get $(\<w_X,b^*a\> - \a_X)\<(b^*a), D(b^*a)\>=  0$.
The algebraic function $\<(b^*a), D(b^*a)\>$ is not zero on $\cp$ (for instance, for
$a=\mathrm{i}e_2, \, b=1$), hence on a nonempty Zariski open subset $\cp_4\subset\cp_3$, we have $\a_X=\<w_X,b^*a\>$.

For $x,y \in \OC$, define the operators $L_x$ and $x \odot y$ on $\OC$ by $L_xq=xq$ and
$(x \odot y) q=\<y,q\>x+\<x,q\>y$. As $L_x^t=L_{x^*}$ and $L_{a^*}L_{b}+L_{b^*}L_{a}=2\<a,b\>\id$, equation
\eqref{eq:diag} can be rewritten as
\begin{equation}\label{eq:diagLie}
    D(a^*b) \odot w_X = \<w_X,1\> [D, L_{a^*}L_{b}] -2 \<\Im \, w_X,b^*a\> D.
\end{equation}
Let $S$ be a symmetric operator commuting with $D$. Multiplying both sides of \eqref{eq:diagLie} by $S$ and taking the
trace we get $\<SD(a^*b), w_X\> = -\<\Im\, w_X,b^*a\> \Tr SD$. Choosing $S$ in such a way that $SD=\Im$ we get
$\<\Im \,w_X,b^*a\>=0$, hence $\<SD(a^*b), w_X\> = 0$, for any symmetric $S$ commuting with $D$. Taking $S$ diagonal
relative to the basis $e_i$ we obtain $\<a^*b, e_i\>\<w_X,e_i\> = 0$, for all $i=1, \dots, 7$. As for a nonempty
Zariski open subset of $X=(a,b) \in \cp$, all the numbers $\<a^*b, e_i\>$ are nonzero, we get $\Im\, w_X=0$,
that is, $w_X=\gamma_X1$, for all $X$ from a nonempty Zariski open subset of $\cp$. Then from the above,
$\a_X=\gamma_X\<a,b\>$ and from \eqref{eq:MXq},
\begin{equation*}
    M_X(q)= \gamma_X(\<1,q\>a^*b - a^*(bq) + \<a,b\> q-2\<a,b\>\<1,q\>+\<b^*a,q\>),
\end{equation*}
for all $X =(a,b)$ from a nonempty Zariski open subset of $\cp$. It then follows from \eqref{eq:defMXq7} that
for all $i=1, \dots, 7$, $\<(\n_XJ_i)X, J_q X \>=\eta_i^{-1} \<M_X(q), e_i\> =
\eta_i^{-1}\gamma_X(\<1,q\>\<a^*b,e_i\> - \<a^*(bq),e_i\> + \<a,b\> \<q,e_i\>)$. Introduce the quadratic forms
$\Phi_i, \Psi_i: \bc^{16} \to \bc^8$ by $(\n_XJ_i)X=(\Phi_i(X), \Psi_i(X))$, for $i =1, \dots, 7$. Then the above
equation and \eqref{eq:Jp7} imply $a^*\Phi_i(X)+b^*\Psi_i(X)=
\eta_i^{-1}\gamma_X(\<a^*b,e_i\> - b^*(ae_i) + \<a,b\> e_i)$. It follows that for every fixed $i =1, \dots, 7$,
the polynomial vectors
$a^*\Phi_i(X)+b^*\Psi_i(X)$ and $T_i(X)=\<a^*b,e_i\> - b^*(ae_i) + \<a,b\> e_i$ are linearly dependent for all
$X=(a,b)$ from a nonempty Zariski open subset of $\cp$, that is, for all $X \in \cp$. Note that
$\<a^*\Phi_i(X)+b^*\Psi_i(X), 1\>=\<T_i(X), 1\>=0$ (the first equation follows from $\<(\n_XJ_i)X,X\>=0$).
Then the rank of the
$7 \times 2$-matrix $N(X)=(a^*\Phi_i(X)+b^*\Psi_i(X)\,|\,T_i(X))$
(whose $j$-th row, $j=1, \dots, 7$, is $\<(a^*\Phi_i(X)+b^*\Psi_i(X),e_j\>\,|\,\<T_i(X),e_j\>)$)
is at most one, for all $X \in \cp$. As
$\mathbf{R}=\bc[X]/\bigl(\|X\|^2\bigr)$, the coordinate ring of $\cp$, is a unique factorization domain
\cite{Nag}, there exist $u_1, u_2 \in \mathbf{R}$ and $v$ in the free module $\mathbf{R}^7$ such that
$\pi(N(X))=(u_1 v \, | \, u_2 v)$, where $\pi:\bc[X] \to \mathbf{R}$ is the natural projection. Let
$U_2, V_j\in \bc[X]$ be the polynomials of the lowest degree in the cosets $\pi^{-1}u_2$ and $\pi^{-1}v_j$
respectively. Lifting the equation $u_2 v_j=\pi(\<T_i(X),e_j\>)=\pi(\<ae_i,be_j\>)$, for $j \ne i$, to $\bc[X]$ we get
$\<ae_i,be_j\>=U_2(X) V_j(X) + \|X\|^2 \Xi_j(X)$, for some $\Xi_j \in \bc[X]$.
Then $U_2$ and $V_j$ are nonzero (as $\<ae_i,be_j\>$ is not divisible by $\|X\|^2=\|a\|^2+\|b\|^2$).
Moreover, as the polynomial on the left-hand side is of degree two in $X$, and $\|X\|^2$ is prime in $\bc[X]$,
the polynomials $U_2$ and $V_j$ are homogeneous, with $\deg U_2+ \deg V_j = 2$ and $\Xi_j$ are constants.

Suppose $\deg U_2=2$. Then the $V_j$'s are nonzero constants for all $j \ne i$. It follows that for some nontrivial
linear combination $e'$ of the $e_j, \; j \ne i$, $\<ae_i, be'\>=0$, for all $a, b \in \Oc$, a contradiction.

Suppose $\deg U_2 = 1$. Then $\deg V_j = 1$, for all $j \ne i$.
Taking $a=0$ we get $\Xi_i=0$ (as the rank of the quadratic form $U_2((0,b)) V_j((0,b))$ in $b$ is at most two),
so for some nonzero linear forms $U_2$ and $V_j$,  $\<ae_i,be_j\>=U_2(X) V_j(X)$. Taking $a=0$ or $b=0$ we obtain that
$U_2(X)=\<l,a\>,\; V_j(X)=\<t_j,b\>$ (or vice versa), for some nonzero $l, t_j \in \Oc, \; j \ne i$. Then taking $b=a$
we arrive at a contradiction.

It follows that $\deg U_2 = 0$, and without loss of generality, we can take $U_2=1$. Then $\pi(N(X))=(u_1 v \, | v)$,
for some $u_1 \in \mathbf{R}, \; v \in \mathbf{R}^7$, so, lifting to $\bc[X]$, we obtain
$a^*\Phi_i(X)+b^*\Psi_i(X) = U_1(X)T_i(X) + \|X\|^2 \Xi(X)$, for some $U_1(X) \in \bc[X]$ and $\Xi(X) \in \bc[X]^8$,
with $\<\Xi(X),1\>=0$.
As the left-hand side is a vector whose components are homogeneous cubic polynomials in $X$, the components of $T_i$
are quadratic forms of $X$, and $\|X\|^2$ is prime in $\bc[X]$, $U_1(X)$ is a linear form and $\Xi$ is a linear
operator. As both sides are real for $X \in \Rs$, for every $i=1, \dots, 7$, there exist $p_i \in \Rs$ and
$\Xi_i: \Rs \to \br^7=\Oc'$ such that $a^*\Phi_i(X)+b^*\Psi_i(X) = \<p_i,X\>T_i(X) + \|X\|^2 \Xi_iX
=\<p_i,X\>(\<a^*b,e_i\> - b^*(ae_i) + \<a,b\> e_i) + \|X\|^2 \Xi_iX$, for all $X=(a,b) \in \Rs$. Let
$\hat\Phi_i(X)=\Phi_i(X)-a \Xi_i X -\frac12 \<p_i,X\>be_i,\;\hat\Psi_i(X)=\Psi_i(X)-b \Xi_i X +\frac12 \<p_i,X\>ae_i$.
Then $a^*\hat\Phi_i(X)+b^*\hat\Psi_i(X) =\<p_i,X\>\<a^*b,e_i\>$. From assertion~2 of Lemma~\ref{l:octlemma}, 
with $N_1(X)=\hat\Phi_i(X), \; N_2(X)=\hat\Psi_i(X)$, $p=p_i, \; u=e_i$, it follows that
$(\n_X J_i)X =(\Phi_i(X), \Psi_i(X))= (a \Xi_i X +\frac12 \<p_i,X\>be_i, b \Xi_i X -\frac12 \<p_i,X\>ae_i)+
\|X\|^2 m_i - \<m_i,X\>X - \sum_{j=1}^7\<J_jX,m_i\>J_jX$, for some $m_i \in \Rs$. As $\<(\n_X J_i)X,X\> =0$
and $\<\Xi_iX,1\>=0$ we get $\<p_i,X\>\<e_i,b^*a\>=0$. Since the polynomial $\<e_i,b^*a\>$ is not divisible by
$\|X\|^2$, we obtain $p_i=0$. Then by \eqref{eq:Jp7},
$(a \Xi_i X, b \Xi_i X)= J_{\Xi_i X}X=\sum_{j=1}^7\<\Xi_iX,e_j\>J_jX$
which implies \eqref{eq:nXJX}, with $b_{ij}=J_jm_i+\Xi_i^te_j$ and with $m_i$ replaced by $\eta_i^{-1} m_i$.
Equation \eqref{eq:bijbji} follows from \eqref{eq:nXJX} and the fact that $\<(\n_XJ_i)X, J_jX\>$ is antisymmetric in
$i$ and $j$.

We next prove \eqref{eq:Qsum}. Substituting \eqref{eq:nXJX} into \eqref{eq:8c} we obtain
\begin{equation}\label{eq:x2ty}
\|X\|^2 T(Y) +\|Y\|^2 T(X)-\<T(X),Y\>Y-\<T(Y),X\>X=0, \quad \text{for all} \; X \perp \cI Y,
\end{equation}
where the quadratic form $T: \Rs \to \Rs$ is defined by
\begin{equation}\label{eq:defT}
T(X)=Q(X)-3\sum\nolimits_{i=1}^7 \<m_i,X\> J_iX.
\end{equation}
For a nonzero $U \in \Rs$, let $X, Y \perp \cI U$ be linearly independent. Then the eight-dimensional spaces
$\cI X$ and $\cI Y$ are both orthogonal to $U$, so their intersection is nontrivial. But if $J_uX=J_vY$,
then $J_{2\<u,v\>u-\|u\|^2v}X=J_{\|v\|^2u}Y$, so $\dim (\cI X \cap \cI Y) \ge 2$. It follows that
$\cI X, \cI Y \perp V$, for some nonzero $V \perp U$. Substituting $V$ for $Y$ into \eqref{eq:x2ty} and taking the
inner product with $U$ we obtain $\|X\|^{-2}\<T(X),U\>=-\|V\|^{-2}\<T(V),U\>$. Similarly, substituting $V$ for $X$
into \eqref{eq:x2ty} and taking the inner product with $U$ we get
$\|Y\|^{-2}\<T(Y),U\>=-\|V\|^{-2}\<T(V),U\>$, so $\|X\|^{-2}\<T(X),U\>=\|Y\|^{-2}\<T(Y),U\>$, for all nonzero
$X, Y \perp \cI U$. It follows that for some function $f: \Rs \to \br$, which is homogeneous of degree one,
\begin{equation}\label{eq:TXU}
\<T(X),U\>=\|X\|^2 f(U), \quad \text{for all} \; X \perp \cI U.
\end{equation}
Taking the inner product of \eqref{eq:x2ty} with $Z \perp \cI Y$
we obtain $\<T(Y),\|X\|^2 Z-\<X,Z\>X\> +\|Y\|^2 \<T(X),Z\>=0$ which by \eqref{eq:TXU} implies
$\<T(X),Z\>=f(\<X,Z\>X)-\|X\|^2 Z)$, for all $X, Z$ with $\dim (\cI X \cup \cI Z) < 16$, that is, with
$\cI X$ and $\cI Z$ having a nontrivial intersection. In particular, taking $Z=J_iX$ we obtain
$\<T(X),J_iX\>=-\|X\|^2 f(J_iX)$. Replacing $X$ by $J_iX$ we get $\<T(J_iX),X\>=-\|X\|^2 f(X)$, for all $X \in \Rs$.
For an arbitrary nonzero $X \in \Rs$, let $U_i, \; i=1, \dots, 8$, be an orthonormal basis for $(\cI X)^\perp$.
Denoting $\Tr T$ the vector in $\Rs$ whose components are the traces of the corresponding components of $T$ and using
the fact that $T(X) \perp X$ (which follows from \eqref{eq:defT} and the fact that $Q(X) \perp X$)
we get $\< \Tr T, X\>=\|X\|^{-2}\sum_{i=1}^7 \<T(J_iX),X\>+\sum_{i=1}^8 \<T(U_i),X\>=
-7 f(X)+8 f(X)=f(X)$, by \eqref{eq:TXU}. Therefore $f$ is a linear form, $f(X)=\<l, X\>$, for some $l \in \Rs$.
Then $T(X)=\|X\|^{-2}\sum_{i=1}^7 \<T(X),J_iX\>J_iX+\sum_{i=1}^8 \<T(X),U_i\>U_i=
-\sum_{i=1}^7 \<l,J_iX\>J_iX+\|X\|^2\sum_{i=1}^8 \<l,U_i\>U_i=\|X\|^2(\pi_{(\cI X)^\perp}l-\pi_{\cI X}l)+\<l,X\>X=
\|X\|^2(l-2\pi_{\cI X}l)+\<l,X\>X$. Substituting this into \eqref{eq:x2ty} and using \eqref{eq:TXU} we get
$l= \pi_{\cI X}l+\pi_{\cI Y}l$, for all $Y \perp \cI X$. Let $l=(l_1, l_2)$. As it follows from \eqref{eq:Jp7},
for $X=(a,b), \; \pi_{\cI X}l=\|X\|^{-2}(\|a\|^2l_1+a(b^*l_2),$ $b(a^*l_1)+\|b\|^2l_2)$, and if $a, b \ne 0$,
the vector $Y=(\|b\|^2aq,-\|a\|^2bq)$ satisfies $Y \perp \cI X$, for all $q \in \Oc$. Then the equation
$l= \pi_{\cI X}l+\pi_{\cI Y}l$ implies $\|q\|^2 a(b^*l_2)=(aq)((bq)^*l_2), \; \|q\|^2 b(a^*l_1)=(bq)((aq)^*l_1)$,
for arbitrary $a, b, q \in \Oc$. The first of them implies $l_2= 0$ (to see that, take $b=q^*$, then the octonions
$a, q, l_2$ associate, for every $a, q$, so $l_2 \in \br$; if $l_2 \ne 0$, then the octonions $a,q,(bq)^*$ associate,
for every $a, q, b$, a contradiction), the second one can be obtained from the first one by interchanging $a$ and $b$,
so it implies $l_1=0$. Thus $l=0$, so $T(X)=0$, which is equivalent to \eqref{eq:Qsum} by \eqref{eq:defT}.

\medskip

2. Substitute $X=J_kY, \; U \perp X, Y$ into \eqref{eq:confBZYX} and
consider the first term in the second summation. As $\<J_iY, X\> = \|Y\|^2 \K_{ik}$, that term equals
$3 \eta_k(2 \< (\n_UJ_k)X, Y \> + \< (\n_XJ_k)Y, U \> + \< (\n_YJ_k)U, X \>)\|Y\|^2$. As $J_k$
is orthogonal and skew-symmetric,
$\< (\n_UJ_k)X, Y \>=\< (\n_UJ_k)J_kY, Y \>=-\< J_k(\n_UJ_k)Y, Y \>=\<(\n_UJ_k)Y, J_kY \>=0$.
Next, $\<(\n_YJ_k)U, X \>=-\<(\n_YJ_k)J_kY, U\>$ $=\<J_k(\n_YJ_k)Y, U\>$ $=
\<(\eta_k^{-1}\|Y\|^2 J_km_k+\sum\nolimits_{j=1}^\nu \<b_{kj},Y\>J_kJ_jY, U\>$ by \eqref{eq:nXJX}. Again by
\eqref{eq:nXJX}, as $Y=-J_kX$, $\< (\n_XJ_k)Y, U \>=\< J_k(\n_XJ_k)X, U \>=
\<J_k(\eta_k^{-1}(\|X\|^2 m_k-\<m_k,X\>X)+\sum\nolimits_{j=1}^\nu \<b_{kj},X\>J_jX),U\>=
\<\eta_k^{-1}\|Y\|^2 J_km_k+\sum\nolimits_{j \ne k} \<b_{kj},J_kY\>J_jY-\<b_{kk},J_kY\>J_kY,U\>$.
Substituting this into \eqref{eq:confBZYX} and using (\ref{eq:Qsum}, \ref{eq:nXJX})
we obtain after simplification:
\begin{equation} \label{eq:XJkY}
\|Y\|^2 (\<2J_km_k,U\>-U(\eta_k))+
 \sum\nolimits_{j=1}^\nu \<\eta_k b_{kj} + \eta_j b_{jk}, \<J_jY,U\>J_kY +\<J_kJ_jY,U\>Y\> =0.
\end{equation}
Choose $i \ne j$ such that $k \ne i,j$
and take the eigenvectors of the symmetric orthogonal operator $J_jJ_iJ_k$ as $U$. For each such
$U, \; J_iJ_kU =\pm J_jU$, so $\dim(\cJ J_kU+ \cJ U) < 2 \nu \le 16$, hence there exists a nonzero
$Y \perp \cJ U + \cJ J_kU$, which implies $U \perp \cJ Y + \cJ J_kY$. Substituting such $U$ and $Y$ into
\eqref{eq:XJkY} we obtain $\<U,2J_km_k-\n\eta_k\>=0$. As the eigenvectors of $J_jJ_iJ_k$ span $\Rs$, equation
\eqref{eq:neta} follows.

Substituting \eqref{eq:neta} into \eqref{eq:XJkY} we obtain
$\sum\nolimits_{j=1}^\nu \<\eta_k b_{kj} + \eta_j b_{jk}, \<J_jY,U\>J_kY +\<J_kJ_jY,U\>Y\> =0$, which implies
$\sum\nolimits_{j \ne k}(\<\eta_k b_{kj} + \eta_j b_{jk}, J_kY\> J_jY+\<\eta_k b_{kj} + \eta_j b_{jk},Y\> J_kJ_jY)=0$.
Equation \eqref{eq:JiJkY} now follows from assertion~2 of \cite[Lemma~3]{Nco}.

By \eqref{eq:defQ} and \eqref{eq:Qsum},
$\<(\n_X \rho)U-(\n_U \rho) X,X\>=3\sum_{i=1}^\nu \<m_i,X\>\<J_iX,U\>$, for all $X, U \in \Rs$. Polarizing this
equation and using the fact that the covariant derivative of $\rho$ is symmetric we obtain
$\<(\n_X\rho)U,Y\>+\<(\n_Y \rho)U,X\>-2\<(\n_U\rho) Y,X\>=3\sum_{i=1}^\nu (\<m_i,Y\>\<J_iX,U\>+\<m_i,X\>\<J_iY,U\>)$.
Subtracting the same equation, with $Y$ and $U$ interchanged, we get $\<(\n_Y \rho)U-(\n_U \rho)Y,X\>=
\sum_{i=1}^\nu (2\<m_i,X\>\<J_iY,U\>$ $+\<m_i,Y\>\<J_iX,U\>- \<m_i,U\>\<J_iX,Y\>)$, which implies \eqref{eq:skewrho}.

To prove \eqref{eq:bijne}, substitute $X \perp \cI Y, \; U = J_kY$ into \eqref{eq:confBZY}. Using
(\ref{eq:lnablarhoi}, \ref{eq:skewrho}) we obtain after simplification:
\begin{equation*}
3 (\n_XJ_k) Y - (\n_YJ_k) X =  -3\eta_k^{-1}\<m_k,Y\>X +
\sum\nolimits_{i=1}^\nu \eta_k^{-1} \<\eta_i b_{ik} + 2 \K_{ik} J_km_k, Y\> J_iX \mod( \cI Y ).
\end{equation*}
Subtracting three times the polarized equation \eqref{eq:nXJX} (with $i=k$) and solving for $(\n_YJ_k) X$ we get
\begin{equation}\label{eq:nyjx}
(\n_YJ_k) X= \sum\nolimits_{i=1}^\nu \tfrac14 \eta_k^{-1}\<3 \eta_k b_{ki}-\eta_i b_{ik}- 2 \K_{ik} J_km_k,Y\>J_iX
\mod( \cI Y ),
\end{equation}
for all $X \perp \cI Y$. Choose $s \ne k$ and define the subset $\mS_{ks} \subset \Rs \oplus \Rs$ by
$\mS_{ks}=\{(X, Y) \, : \, X, Y \ne 0$, $X, J_kX, J_sX \perp \cJ Y \}$. It is easy to see that
$(X,Y) \in \mS_{ks} \Leftrightarrow (Y,X) \in \mS_{ks}$ and that replacing $\cJ Y$ by $\cI Y$ in the definition of
$\mS_{ks}$ gives the same set $\mS_{ks}$. Moreover, the set $\{X \, : \, (X,Y) \in \mS_{ks}\}$ (and hence the set
$\{Y \, : \, (X,Y) \in \mS_{ks}\}$) spans $\Rs$. If $\nu<8$, this follows from \cite[Lemma~3.2 (4)]{Nhjm}.
If $\nu=8$, the Clifford structure is given by \eqref{eq:Jp}. Take $X=(a,b)$, with $\|a\|=\|b\|=1$, and
$Y=(bu, au^*)$ for some nonzero $u \in \Oc$. Then the condition $X \perp \cJ Y$ is satisfied and the condition
$J_kX \perp \cJ Y$ is equivalent to $\<(ae_k^*)q,bu\>+\<(be_k)q^*,au^*\>=0$, for all $q \in \Oc$, that is, to
$(ae_k^*)^*(bu)+(au^*)^*(be_k)=0$. As $(ae_k^*)^*(bu)+(au^*)^*(be_k)=
2\<ae_k^*,b\>-b^*((ae_k^*)u)+2\<au^*,b\>-b^*((au^*)e_k)=2\<ae_k^*,b\>-2\<e_k,u\>b^*a+2\<au^*,b\>$, the latter
condition is satisfied, if we choose $a, b$ and $u$ in such a way that $b^*a, u$ and $e_k$ are orthogonal.
Similar arguments for $e_s$ show that for every $X=(a,b)$, with $\|a\|=\|b\|=1$ and $b^*a \perp e_k, e_s$,
there exists $Y \ne 0$ such that $(X,Y) \in \mS_{ks}$. In particular, taking $X=(\pm be_i, b)$, with a fixed
$e_i \perp e_k,e_s$ and arbitrary unit $b \in \Oc$ we obtain that the set $\{X \, : \, (X,Y) \in \mS_{ks}\}$
spans $\Rs$.

Now, for $(X, Y) \in \mS_{ks}$, take the inner product of \eqref{eq:nyjx} with $J_sX$. Since $\<(\n_YJ_k)X,J_sX\>$ is
antisymmetric in $k$ and $s$, we get $\<(3-\eta_k\eta_s^{-1})b_{ks}+(3-\eta_s\eta_k^{-1})b_{sk}, Y\>=0$, for a set of
the $Y$'s spanning $\Rs$. So $(3-\eta_k\eta_s^{-1})b_{ks}+(3-\eta_s\eta_k^{-1})b_{sk}=0$, for all $k \ne s$. This
and \eqref{eq:JiJkY} imply \eqref{eq:bijne}.

Now from (\ref{eq:JiJkY}, \ref{eq:bijne}) it follows that
$b_{ij}+b_{ji}=0$ for all $i \ne j$, so by \eqref{eq:bijbji}, $\eta_i^{-1}J_jm_i=-\eta_j^{-1}J_im_j$.
Acting by $J_iJ_j$ we obtain that the vector $\eta_i^{-1}J_im_i$ is the same, for all $i=1, \dots, \nu$,
which proves \eqref{eq:allequal}.
\end{proof}

\begin{lemma}\label{l:codazzi}
In the assumptions of Lemma~\ref{l:locc1}, let $x \in M'$ and let $\mU$ be the neighborhood of $x$ introduced in
assertion~\eqref{it:cl} of Lemma~\ref{l:locc1}. Then there exists a smooth metric on $\mU$ conformally
equivalent to the original metric whose curvature tensor has the form \eqref{eq:confcs}, with $\rho$ a constant
multiple of the identity.
\end{lemma}

\begin{proof}
If $\nu \le 4$, the proof follows from \cite[Lemma~7]{Nco}. Suppose $\nu \ge 4$.
Let $f$ be a smooth function on $\mathcal{U}$ and let $\<\cdot, \cdot\>' = e^f\<\cdot, \cdot\>$. Then
$W'=W, \; J_i'=J_i, \; \eta_i'= e^{-f} \eta_i$ and, on functions, $\n'=e^{-f}\n$, where we use the dash for the
objects associated to metric $\<\cdot, \cdot\>'$. Moreover, the curvature tensor $R'$ still has the form
\eqref{eq:confcs}, and all the identities of Lemma~\ref{l:nablarho} remain valid.

In the cases considered in Lemma~\ref{l:nablarho}, the ratios $\eta_i/\eta_1$ are constant, as it follows from
(\ref{eq:neta},\ref{eq:allequal}). In particular, taking $f=\ln|\eta_1|$ we obtain that $\eta_1'$ is a constant,
so all the $\eta_i'$ are constant, $m_i'=0$ by \eqref{eq:neta}, so $(\n'_Y \rho')U-(\n'_U \rho')Y=0$
by \eqref{eq:skewrho}. Dropping the dashes, we obtain that, up to a conformal smooth change of the metric on
$\mU$, the curvature tensor has the form \eqref{eq:confcs}, with $\rho$ satisfying the identity
$(\n_Y \rho)X=(\n_X \rho)Y$, for all $X, Y$, that is, with $\rho$  being a symmetric \emph{Codazzi tensor}.

Then by \cite[Theorem~1]{DS}, at every point of $\mathcal{U}$, for any three eigenspaces
$E_\b, E_\gamma, E_\a$ of $\rho$, with $\a \notin \{\b, \gamma\}$,
the curvature tensor satisfies $R(X,Y)Z = 0$, for all $X \in E_\b, \; Y \in E_\gamma, \; Z \in E_\a$.
It then follows from \eqref{eq:confcs} that
\begin{equation}\label{eq:codazzi}
\begin{gathered}
\sum\nolimits_{i=1}^\nu \eta_i (2 \< J_iX, Y \> J_iZ + \< J_iZ, Y \> J_iX - \< J_iZ, X \> J_iY) = 0, \\
\text{for all } X \in E_\b, \; Y \in E_\gamma, \; Z \in E_\a, \quad \a \notin \{\b, \gamma\}.
\end{gathered}
\end{equation}
Suppose $\rho$ is not a multiple of the identity. Let $E_1, \ldots, E_p, \; p \ge 2$, be the eigenspaces of $\rho$.
If $p>2$, denote $E_1'=E_1, \; E_2'=E_2 \oplus \dots \oplus E_p$. Then by linearity, \eqref{eq:codazzi}
holds for any $X, Y \in E'_\a, \; Z \in E'_\b$, such that $\{\a, \b\}=\{1,2\}$. Hence to prove
the lemma it suffices to show that \eqref{eq:codazzi} leads to a contradiction,
in the assumption $p=2$. For the rest of the proof, suppose that $p=2$. Denote $\dim E_\a =d_\a$.

If $\nu <8$, the claim follows from the proof of \cite[Lemma~7]{Nco} (see \cite[Remark~4]{Nco}). Suppose $\nu=8$,
then the Clifford structure is given by \eqref{eq:Jp}.

Choosing $Z \in E_\a, \; X, Y \in E_\b,\; \a \ne \b$, and taking the inner product of \eqref{eq:codazzi} with $X$ we
obtain $\sum_{i=1}^8 \eta_i \< J_iX, Y \>\< J_iX, Z\> = 0$.
It follows that for every $X \in E_\a$, the subspaces $E_1$ and $E_2$ are invariant subspaces of the symmetric
operator $R'_X \in \End(\Rs)$ defined by $R'_XY=\sum_{i=1}^8 \eta_i \< J_iX, Y \>J_iX$. So
$R'_X$ commutes with the orthogonal projections $\pi_\b: \Rs \to E_\b, \; \b=1,2$. Then for all $\a,\b=1,2$
($\a$ and $\b$ can be equal), all $X \in E_\a$ and all $Y \in \Rs, \quad
\sum_{i=1}^8 \eta_i \< J_iX, \pi_\b Y \> J_iX = \sum_{i=1}^8 \eta_i \< J_iX, Y \> \pi_\b J_iX$.
Taking $Y=J_jX$ we get that $\pi_\b J_jX \subset \cJ X$, that is, $\pi_\b \cJ X \subset \cJ X$, for all
$X \in E_\a, \; \a,\b=1,2$. As $\pi_1 + \pi_2 = \id$, we obtain
$\cJ X \subset \pi_1 \cJ X \oplus \pi_2 \cJ X \subset \cJ X$, hence
\begin{equation}\label{eq:sumproj}
\cJ X = \pi_1 \cJ X \oplus \pi_2 \cJ X, \quad \text{for all} \;  X \in E_1 \cup E_2.
\end{equation}
As all four functions $f_{\a\b}: E_\a \to \mathbb{Z}, \; \a,\b=1,2$, defined by
$f_{\a\b}(X) = \dim \pi_\b \cJ X, \; X \in E_\a$, are lower semi-continuous and $f_{\a 1}(X) + f_{\a 2}(X) = 8$
for all nonzero $X \in E_\a$, there exist constants $c_{\a\b}$, with $c_{\a 1} + c_{\a 2} = 8$, such that
$\dim \pi_\b \cJ X = c_{\a\b}$, for all $\a,\b=1,2$ and all nonzero $X \in E_\a$.

Consider two cases.

First assume that there exist no nonzero $Y \in E_\a, \; Z \in E_\b, \; \a \ne \b$, such that $Y \perp \cJ Z$. Then it
follows from \eqref{eq:sumproj} that $\cJ X \supset E_\b$, for $\a \ne \b$. Then $d_1, d_2 \le 8$. As $d_1+d_2 = 16$,
we obtain $d_1=d_2 = 8$ and $\cJ X = E_\b$, for every $X \in E_\a, \; \a \ne \b$. Then \eqref{eq:codazzi}, with
$X,Y \in E_\a, \; Z \in E_\b$, $\a \ne \b$, gives $\sum_{i=1}^8 \eta_i(\< Z, J_iX \> J_iY - \< Z, J_iY \> J_iX) = 0$.
Taking the inner product with $J_jX$ we get
$\eta_j\< Z, J_jY \> \|X\|^2 = \sum_{i=1}^8 \eta_i \< Z, J_iX \> \<J_iY,J_jX\>$. Taking $Z = J_kX \, (\in E_\b)$ and
assuming $X \perp Y$ we obtain $(\eta_j+\eta_k)\< X, J_kJ_jY \>  = 0$, for $k \ne j$. Taking
$X = J_kJ_jY \, (\in E_\a)$ we get $\eta_j+\eta_k=0$, for $k \ne j$, so all the $\eta_i$'s are zeros, a
contradiction with $\nu = 8$.

Otherwise, assume that there exist nonzero $Y \in E_\a, \; Z \in E_\b, \; \a \ne \b$, such that $Y \perp \cJ Z$.
Substituting such $Y$ and $Z$ into \eqref{eq:codazzi}, with $X \in E_\a$, we obtain
$\sum_{i=1}^8 \eta_i (2 \< J_iY, X \> J_iZ + \< J_iZ, X \> J_iY) = 0$. Taking $X\in E_\a$ orthogonal to $\pi_\a \cJ Y$
(and then to $\pi_\a \cJ Z$) we obtain $\pi_\a \cJ Y=\pi_\a \cJ Z$.
As the condition $Y \perp \cJ Z$ is symmetric in $Y$ and $Z$, we can interchange $Y$ and $Z$ and $\a$ and $\b$
to get $\pi_\b \cJ Y=\pi_\b \cJ Z$, which by \eqref{eq:sumproj}
implies that $\cJ Y = \cJ Z$, for any two nonzero vectors $Y \in E_\a, \; Z \in E_\b, \; \a \ne \b$, such that
$Y \perp \cJ Z$. Now, if for some nonzero $Y \in E_\a$ there exists $Z \in E_\b, \; \a \ne \b$, such that
$Y \perp \cJ Z$, then by \eqref{eq:sumproj}, the space $\pi_\b \cJ Y$ is a proper subspace of $E_\b$, so
$c_{\a\b} < d_\b$, which implies that for every nonzero $X \in E_\a$ and any nonzero $Z$ from the orthogonal
complement to $\pi_\b \cJ X$ in $E_\b$ (which is nontrivial), $X \perp \cJ Z$, hence $\cJ X = \cJ Z$, from the above.

Consider an operator $P= \prod_{i=1}^8 J_i$. As the $J_i$'s are anticommuting almost Hermitian structures,
$P$ is symmetric and orthogonal, and $\Tr P =0$, as $P$ is the product of the
symmetric operator $\prod_{i=1}^7 J_i$ and the skew-symmetric one, $J_8$. So its eigenvalues are $\pm 1$, with
both eigenspaces $V_{\pm}$ of dimension $8$. As the $J_i$'s anticommute, each of them interchanges the eigenspaces
of $P: \; J_i V_{\pm}=V_{\mp}$.

From the above, for every unit vector $X \in E_1$, there exists a unit vector $Z \perp X$ such that $\cJ X = \cJ Z$.
Therefore, by  \eqref{eq:ortmult} there exists $F: \Ro \to \Ro$ such that for all $u \in \Ro,\; J_{F(u)}X=J_uZ$. As
the right-hand side is linear in $u$, the map $F$ is also linear: $F(u)=Au$, for some $A \in \End(\Ro)$.
Moreover, as $J_uX$ is an orthogonal multiplication and as $X$ and $Z$ are orthonormal, the operator $A$ is
orthogonal and skew-symmetric. Without losing generality, we can assume that an orthonormal basis for $\Ro$ is
chosen in such a way that $J_iX=J_{i+4}Z$, for $i=1,2,3,4$, so $J_jJ_{j+4}J_{i+4}J_iX=X$, for all
$1 \le i \ne j \le 4$. As changing an orthonormal basis for $\Ro$ may only change the sign of $P$, it
follows that $X$ is an eigenvector of $P$, say $X \in V_+$. As $X$ is an arbitrary unit vector from $E_1$,
by continuity, $E_1 \subset V_+$. Similarly, $E_2 \subset V_-$. But as every $J_i$ interchanges the $V_{\pm}$'s,
we get $\cJ E_1 = E_2$, which contradicts the assumption that there exist nonzero $Y \in E_1, \; Z \in E_2$ with
$Y \perp \cJ Z$.

Hence the Codazzi tensor $\rho$ is a multiple of the identity. The definition of the Codazzi tensor easily implies
that $\rho$ is a constant multiple of the identity on $\mU$.
\end{proof}

By Lemma~\ref{l:codazzi}, up to a conformal change of the metric, we can assume that on $\mU$, the curvature tensor
has the form \eqref{eq:confcs}, with $\rho$ a constant multiple of the identity.
Then \eqref{eq:defQ} implies that $Q=0$, so $m_i=0$ by \eqref{eq:Qsum}, $\n \eta_i=0$ by \eqref{eq:neta} and
$(\n_X J_i)X = \sum_{j=1}^\nu \<b_{ij},X\>J_jX$ by \eqref{eq:nXJX}. Then from \eqref{eq:confcs},
$(\n_XR)(X,Y)X=3\sum_{i=1}^\nu \eta_i( \< (\n_XJ_i)X, Y \> J_iX + \< J_iX, Y \> (\n_XJ_i)X)=
3\sum_{i,j=1}^\nu \<\eta_ib_{ij}+\eta_j b_{ji},X\>$ $\< J_jX, Y \> J_iX=0$,
as $\eta_ib_{ij}+\eta_jb_{ji}=0$ (by \eqref{eq:JiJkY} for $i\ne j$, and by \eqref{eq:bijbji} for $i=j$, as $m_i=0$).

It is well-known \cite[Proposition~2.35]{Bes} that the equation $(\n_XR)(X,Y)X=0$ implies $\n R=0$, so the
metric on $\mU$ is locally symmetric. As $\rho$ is a multiple of the identity, it follows from \eqref{eq:confcs} that
the curvature tensor is Osserman. Then by \cite[Lemma 2.3]{GSV}, $\mU$ is either flat or is locally isometric to a
rank-one symmetric space.

Thus, for every $x \in M'$ satisfying assertion~\eqref{it:cl} of Lemma~\ref{l:locc1}, the metric on the neighborhood
$\mU=\mU(x)$ is either conformally flat or is conformally equivalent to the metric of a rank-one symmetric space.

\subsection{Cayley case}
\label{ss:oc}

In this section, we consider the case when the Weyl tensor has a Cayley structure.

Let $x \in M'$ and let $\mU= \mU(x)$ be neighborhood of $x$ defined in assertion~\eqref{it:oc} of
Lemma~\ref{l:locc1}. Up to a conformal change of the metric, the curvature tensor on $\mU$ is given by
\eqref{eq:op2ctc}. Then
\begin{equation}\label{eq:op2nctc}
    (\n_Z R)(X,Y)=(\n_Z \rho) X \wedge Y - (\n_Z \rho) Y \wedge X +
    \ve \sum\nolimits_{i=0}^8 ((\n_Z S_i) X \wedge S_i Y + S_i X \wedge (\n_Z S_i) Y),
\end{equation}
so the second Bianchi identity has the form
\begin{equation}\label{eq:op2bi}
    \sigma_{XYZ}\bigl(A(Y,X) \wedge Z + \ve \sum\nolimits_{i=0}^8 S_i Z \wedge B_i(Y, X)\bigr)=0,
\end{equation}
where $\sigma_{XYZ}$ is the sum taken over the cyclic permutations of $(X, Y, Z)$, and the skew-symmetric maps
$A, B_w: \Rs \times \Rs \to \Rs$ are defined by
\begin{equation}\label{eq:abw}
    A(Y,X)=(\n_Y \rho) X - (\n_X \rho) Y, \quad B_w(Y,X)=(\n_Y S_w) X - (\n_X S_w) Y,
\end{equation}
for $w \in \Rd, \; X,Y \in \Rs$, where $S_wX=\sum_{i=0}^8 w_i S_i$ is the orthogonal multiplication defined in
Section~\ref{ss:actcayley}.

\begin{lemma}\label{l:op2bi}
In the assumptions of Lemma~\ref{l:locc1}, let $x \in M'$ and let $\mU$ be the neighborhood of $x$ introduced in
assertion~\eqref{it:oc} of Lemma~\ref{l:locc1}. For every point $y \in \mU$, identify
$T_yM^{16}$ with the Euclidean space $\Rs$ via a linear isometry. Then

\begin{enumerate}[1.]
   \item \label{it:0op2bi}
    There exists a linear map $N:\Rs \to \Sk(\Rs), \, X \to N_X$, such that $\n_XS_w=[S_w,N_X]$ for all
    $w \in \Rd, \, X \in \Rs$.

  \item
  For every unit vector $w \in \Rd$, there exists a linear operator $L_w: w^\perp \to E_1(S_w)$ such that for every
  $X, Y \in E_1(S_w), \; Z \in E_{-1}(S_w)$, and every $u \in \Rd, \, u \perp w$,
  \begin{equation}\label{eq:pi1AB}
  \begin{gathered}
    \pi_{E_1(S_w)}B_w(Y,X)=0, \\
    \pi_{E_1(S_w)}A(Y,X)=0, \\ 
    \pi_{E_1(S_w)}B_u(Y,X)=\<L_wu, X\>Y-\<L_wu, Y\>X, \\
    \pi_{E_1(S_w)}(A(Z, Y)- \ve B_w(Z, Y))=\ve S_{L_w^tY}Z,
  \end{gathered}
  \end{equation}
  where $E_{\pm1}(S_w)$ are the $\pm 1$-eigenspaces of $S_w$ and $\pi_{E_1(S_w)}$ is the orthogonal projection to
  $E_1(S_w)$. 

  \item \label{it:2op2bi}
  There exists a bilinear skew-symmetric map $T: \Rd \times \Rd \to \Rs$ such that for all
  $X, Y \in \Rs, \; w \in \Rd$,
  \begin{equation*} 
    A(Y, X) = 0, \quad B_w(Y, X) = \sum\nolimits_{i=0}^8 (\<T(w,e_i), X\> S_i Y - \<T(w,e_i), Y\> S_i X).
  \end{equation*}

  \item \label{it:3op2bi}
  The tensor $\rho$ is a constant multiple of the identity on $\mU$.

  \item \label{it:4op2bi}
  $\n_X S_w = - \sum\nolimits_{i=0}^8 \<T(w,e_i), X\> S_i$.
\end{enumerate}
\end{lemma}
\begin{proof}
1. Denote $T_i=\n_XS_i$. The operators $T_i$ are symmetric and satisfy $T_iS_j+T_jS_i+S_jT_i+S_iT_j=0$, for all
$i,j=0,1, \dots, 8$, by \eqref{eq:SiSj}. In particular, the operators $S_iT_i$ are skew-symmetric and
\begin{equation}\label{eq:[sst]}
    [S_i,S_iT_i]=2T_i, \qquad [S_i,S_jT_j]=T_i+S_jT_iS_j, \; i \ne j.
\end{equation}
Define $N:\Rs \to \Sk(\Rs)$ by $N_X=\frac{1}{48} (7\id + \mA)\sum_{j=0}^8 S_j T_j$, for $X \in \Rs$. The fact that
$N_X$ is indeed skew-symmetric follows from assertion~2 of Lemma~\ref{l:mA}, as $S_jT_j$ are skew-symmetric. Moreover,
for any $i=0, \dots, 8$, from assertion~3 of Lemma~\ref{l:mA},
$[S_i,\mA(Q)]=-2[S_i,Q]-\mA([S_i,Q])=(-2\id-\mA)([S_i,Q])$,
so $[S_i,N_X]=\frac{1}{48}[S_i,(7\id + \mA)\sum_{j=0}^8 S_j T_j]=\frac{1}{48}(5\id - \mA)\sum_{j=0}^8 [S_i,S_j T_j]$.
Then by \eqref{eq:SiSj} and \eqref{eq:[sst]},
$[S_i,N_X]=\frac{1}{48}(5\id - \mA)(11\id + \mA)T_i=-\frac{1}{48}(\mA^2+6\mA-55\id)T_i$. As $T_i$ is symmetric and
$\Tr T_i=0$ (which follows from $T_i=S_i \cdot S_iT_i$ and the fact that $S_iT_i$ is skew-symmetric),
we obtain from assertion~1 of Lemma~\ref{l:mA} that $T_i \in \mL_1 \oplus \mL_4$. Then by assertion~2 of
Lemma~\ref{l:mA}, $(\mA - \id)(\mA + 7\id)T_i=0$, which implies $[S_i,N_X]=T_i$.

2. As by assertion~2 of Lemma~\ref{l:mA}, the operator $\mA$ does not depend on the choice of the orthonormal basis
$\frac14 S_i$ for $\mL_1$, it follows from \eqref{eq:op2nctc} that we lose no generality by assuming $w=e_0 \in \Rd$.
By assertion~1, $\n_XS_0=[S_0,N_X]$, so from \eqref{eq:abw},
$B_0(Y,X)=[S_0,N_Y]X-[S_0,N_X]Y=(S_0-\id)(N_YX-N_XY)=-2\pi_{E_{-1}(S_0)}(N_YX-N_XY)$, so
$\pi_{E_1(S_0)}B_0(Y,X)=0$. This proves the first identity of \eqref{eq:pi1AB}.

To prove the second one, take $X, Y, Z \in E_1(S_0)$. Note that by \eqref{eq:SiSj}, $S_jX \in E_{-1}(S_0)$ for all
$j \ne 0$ (and similarly, for $Y$ and $Z$). Then projecting \eqref{eq:op2bi} to $E_1(S_0) \wedge E_1(S_0)$ and using
the first identity of \eqref{eq:pi1AB}, we obtain $\sigma_{XYZ}((\pi_{E_1(S_0)}A(Y,X)) \wedge Z)=0$. Assuming $X,Y,Z$
linearly independent and acting by the both sides on a vector $U \in E_1(S_0) \cap (\Span(X,Y,Z))^\perp$ we obtain
$\<A(Y,X), U\>=0$, so $\pi_{E_1(S_0)}A(Y,X)\in\Span(X,Y)$ (as $Z \in E_1(S_0) \setminus\Span(X,Y)$ can be chosen
arbitrarily). Then for orthonormal vectors $X, Y \in E_1(S_0),\; \pi_{E_1(S_0)}A(Y,X)=\<A(Y,X),X\>X+\<A(Y,X),Y\>Y$,
so the coefficient of $X \wedge Z$ in $\sigma_{XYZ}((\pi_{E_1(S_0)}A(Y,X)) \wedge Z)=0$ (with orthonormal
$X, Y, Z \in E_1(S_0)$) gives $\<A(Y,X),X\>+\<A(Y,Z),Z\>=0$ (using the fact that $A(Y,X)$ is antisymmetric in $X, Y$,
by \eqref{eq:abw}). Then $\<A(Y,X),X\>=0$, hence $\pi_{E_1(S_0)}A(Y,X)=0$, for any orthonormal vectors
$X, Y \in E_1(S_0)$.

For the remaining two identities, take $X, Y \in E_1(S_0), \; Z \in E_{-1}(S_0)$ in \eqref{eq:op2bi} and
project the resulting equation to $E_1(S_0) \wedge E_1(S_0)$. As by \eqref{eq:SiSj},
$S_iX, S_iY \in E_{-1}(S_0), \; S_iZ \in E_1(S_0)$,
for all $i \ge 1$, we get
\begin{equation}\label{eq:AepsB}
\ve \sum\nolimits_{i=1}^8 S_i Z \wedge \pi_{E_1(S_0)}(B_i(Y, X))+
\pi_{E_1(S_0)}((A-\ve B_0)(Z,Y)) \wedge X + \pi_{E_1(S_0)}((A- \ve B_0)(X,Z))\wedge Y =0.
\end{equation}
Taking the inner product of \eqref{eq:AepsB} with $S_jZ \wedge S_kZ$ we find that the expression
$\<\ve \|Z\|^2B_j(Y, X)- \<S_j Z,X\>(A-\ve B_0)(Z,Y)-\<S_j Z,Y\>(A- \ve B_0)(X,Z), S_kZ\>$ is symmetric in
$j,k \ge 1$, for all $X, Y \in E_1(S_0)$, $Z \in E_{-1}(S_0)$. Fix $j,k \ge 1,\, j \ne k$, and take
$Z \perp \Span_{a=j,k}(S_aX,S_aY)$. Then
$S_k \pi_{E_1(S_0)}B_j(Y, X)-S_j \pi_{E_1(S_0)}B_k(Y, X) \in \Span_{a=j,k}(S_aX,S_aY)$, so
(acting by $S_jS_k$ on the both sides)
$S_j \pi_{E_1(S_0)}B_j(Y, X)+S_k \pi_{E_1(S_0)}B_k(Y, X) \in \Span_{a=j,k}(S_aX,S_aY)$. Let
$\mS_{jk} \subset E_1(S_0) \times E_1(S_0)$ be the set of pairs $(X, Y)$ such that
$X \ne 0, \; Y \notin \Span(X, S_jS_k X)$. Then $\mS_{jk}$ is open and dense in $E_1(S_0) \times E_1(S_0)$,
and the vectors $S_jX,S_kX,S_jY,S_kY$ are linearly independent for $(X,Y)\in\mS_{jk}$. As
$S_j\pi_{E_1(S_0)}B_j(Y, X)+S_k\pi_{E_1(S_0)}B_k(Y, X)$ is skew-symmetric in $X, Y$ and symmetric in $k, j$,
there exist (rational) functions $f_{jk}, f_{kj}:\mS_{jk} \to \br$ such that
$S_j \pi_{E_1(S_0)}B_j(Y, X)+S_k \pi_{E_1(S_0)}B_k(Y, X)=
f_{jk}(X,Y)S_jX+f_{kj}(X,Y)S_kX-f_{jk}(Y,X)S_jY-f_{kj}(Y,X)S_kY$, for every $(X,Y)\in\mS_{jk}$. Taking
$Y' \in  E_1(S_0) \setminus \Span(X,S_jS_kX,Y,S_jS_kY)$ (so that the vectors $S_jX,S_kX,S_jY,S_kY,S_jY',S_kY'$ are
linearly independent) and replacing $Y$ by $aY+bY' \ne 0$, so that $(X, aY+bY') \in \mS_{jk}$,
we obtain from the linearity of the left-hand side that $f_{jk}$ and $f_{kj}$ do not depend
on the first argument and are linear in the second one. It follows that for some vectors
$v_{jk} \in E_1(S_0), \; 1 \le j \ne k \le 8, \; S_j \pi_{E_1(S_0)}B_j(Y, X)+S_k \pi_{E_1(S_0)}B_k(Y, X)=
\<v_{jk},Y\>S_jX+\<v_{kj},Y\>S_kX-\<v_{jk},X\>S_jY-\<v_{kj},X\>S_kY$, for all $X, Y \in E_1(S_0)$. Choose
$i, l$ such that $i,j,k,l$ are all distinct and add to the above equation the same one with $j,k$ replaces by $i,l$.
The left-hand side of the resulting equation is symmetric in all four indices $i,j,k,l$, hence the right-hand side
also is. Choosing $X, Y \in E_1(S_0)$ in such a way that the eight vectors $S_aX,S_aY, \; a=i,j,k,l$, are
linearly independent (to do that, take $X \ne 0$ and
$Y \notin \Span_{a \ne b, \{a,b\} \subset \{i,j,k,l\}}(X, S_aS_b X)$), we obtain that $v_{jk}=v'_j$, for all
$j,k \ge 1, \, j\ne k$, with some $v'_j \in E_1(S_0)$. It follows that
$S_j (\pi_{E_1(S_0)}B_j(Y, X)-\<v'_j,Y\>X+\<v'_j,X\>Y)+S_k (\pi_{E_1(S_0)}B_k(Y, X)-\<v'_k,Y\>X+\<v'_k,X\>Y)=0$, for
all $X, Y \in E_1(S_0)$ and all $j,k \ge 1, \, j \ne k$, which implies
$\pi_{E_1(S_0)}B_j(Y, X)=\<v'_j,Y\>X-\<v'_j,X\>Y$. This proves the third identity of \eqref{eq:pi1AB}, if we define
the operator $L_{e_0}$ by $L_{e_0}e_j=-v'_j$ and extend it by linearity to $e_0^\perp$.

Substituting the third identity of \eqref{eq:pi1AB}, with $w=e_0$, to \eqref{eq:AepsB} we obtain
$(F(Z)X) \wedge Y=(F(Z)Y) \wedge X$, where the linear operator $F:E_{-1}(S_0) \to \End(E_1(S_0))$
is defined by $F(Z)X=\ve \sum\nolimits_{i=1}^8 \<L_{e_0}e_i, X\> S_i Z + \pi_{E_1(S_0)}((A- \ve B_0)(X,Z))$,
for $Z \in E_{-1}(S_0), \; X \in E_1(S_0)$. It follows that $F=0$, so
$\pi_{E_1(S_0)}((A- \ve B_0)$ $(X,Z))=-\ve \sum\nolimits_{i=1}^8 \<e_i, L_{e_0}^tX\> S_i Z= -\ve S_{L_{e_0}^tX} Z$,
which proves the fourth identity of \eqref{eq:pi1AB}.

3. For a unit vector $w \in \Rd$, extend the operator $L_w$ from $w^\perp$ to $\Rd$ by linearity putting $L_ww=0$,
and then define $L_w: \Rd \to \Rs$, for all $w \ne 0$, by $L_w=L_{w/\|w\|}$. The identities \eqref{eq:pi1AB} then
hold for all $u, w \in \Rd, \; w \ne 0$, if we replace $E_1(S_w)$ by $E_{\|w\|}(S_w)\,(=E_1(S_{w/\|w\|})$.
Combining the first and the third identities of \eqref{eq:pi1AB} we obtain that for all $u, w \in \Rd, \; w \ne 0$,
and all $X, Y \in E_{\|w\|}(S_w)$,
\begin{equation*}
    \pi_{E_{\|w\|}(S_w)}B_u(Y,X)=\<L_wu, X\>Y-\<L_wu, Y\>X.
\end{equation*}
For every $u \in \Rd$, define the quadratic map $Q_u:\Rs \to \Rs$ by $\<Q_u(Y), X\>=\<B_u(Y,X),Y\>$. Taking the
inner product of the above equation with $Y \in E_{\|w\|}(S_w)$ and then integrating by $Y$ over the unit sphere
$\mathbb{S}(w) \subset E_{\|w\|}(S_w)$ we obtain $\<\int_{\mathbb{S}(w)}Q_u(Y)dY,X\>=\frac{7\omega_7}{8}\<L_wu, X\>$,
for all $X \in E_{\|w\|}(S_w)$, where $\omega_7$ is the volume of $\mathbb{S}(w)$. Relative to some orthonormal
basis $\{E_i\}$ for $\Rs$, the $i$-th component of $\int_{\mathbb{S}(w)}Q_u(Y)dY \in \Rs$ is
$\frac{\omega_7}{8}\Tr (\pi_{E_{\|w\|}(S_w)}Q_{u,i})$, where $Q_{u,i} \in \Sym(\Rs)$ is
the operator associated to the quadratic form $Y \to \<Q_u(Y),E_i\>$.
As $\pi_{E_{\|w\|}(S_w)}=\frac{1}{2}(\id+\|w\|^{-1}S_w)$, we obtain
$\int_{\mathbb{S}(w)}Q_u(Y)dY= \frac{7\omega_7}{8} (Cu + \|w\|^{-1}T(u,w))$, where a linear operator $C:\Rd \to \Rs$
and a bilinear map $T:\Rd \times \Rd \to \Rs$ are defined by $\<Cu,E_i\>=\frac1{14} \Tr Q_{u,i}$ and
$\<T(u,w),E_i\>=\frac1{14} \Tr (Q_{u,i}S_w)$, for $1 \le i \le 16$. It follows that
$\<Cu + \|w\|^{-1}T(u,w), X\>=\<L_wu, X\>$, for all $X \in E_{\|w\|}(S_w)$, which gives
\begin{equation*}
    \pi_{E_{\|w\|}(S_w)}(B_u(Y,X)-\<Cu + \|w\|^{-1}T(u,w), X\>Y+\<Cu + \|w\|^{-1}T(u,w), Y\>X)=0,
\end{equation*}
for all $X, Y \in E_{\|w\|}(S_w)$. This equation is satisfied, if we substitute
$B_u'(Y,X)=\<Cu, X\>Y-\<Cu, Y\>X+\sum_{i=0}^8(\<T(u,e_i),X\>S_iY-\<T(u,e_i),Y\>S_iX)$ for $B_u(Y,X)$ (this follows
from the fact that the sum on the right-hand side of $B_u'(Y,X)$ does not depend on the choice of an orthonormal basis
$\{e_i\}$ for $\Rd$, so we can take $e_0=\|w\|^{-1}w$; then $S_iX, S_iY \in E_{-\|w\|}(S_w)$ for $i \ne 0$,
by \eqref{eq:SiSj}). Therefore, for every $u \in \Rd$, the bilinear skew-symmetric map
$B''_u:\Rs \times \Rs \to \Rs$ defined by $B''_u=B_u-B_u'$ satisfies the hypothesis of assertion~5 of
Lemma~\ref{l:mA}.
It follows that for some $q:\Rd \to \Rs, \quad B''_u(Y,X)=(\mA-\id)(X \wedge Y)q(u)$. As the left-hand
side is linear in $u$, the map $q$ is a linear, $q(u)=C'u$, so for all $X, Y \in \Rs, \; u \in \Rd$,
\begin{equation}\label{eq:BuXY}
B_u(Y,X)=\sum\nolimits_{i=0}^8 S_i(X \wedge Y)(T(u,e_i)+S_iC'u)+(X \wedge Y)(C-C')u.
\end{equation}
From the first identity of \eqref{eq:pi1AB} it now follows that $\<T(u,u)+\|u\|Cu,X\>=0$, for all
$u \in \Rd, \; X \in E_{\|u\|}(S_u)$. As
$\pi_{E_{\|u\|}(S_u)}=\frac12 (\id+\|u\|^{-1}S_u)$, this implies $\|u\|(T(u,u)+S_uCu)+(S_uT(u,u)+\|u\|^2Cu)=0$. Since
$\|u\|$ is not a rational function, we obtain $T(u,u)+S_uCu=0$, for all $u \in \Rd$, so the bilinear map
$T':\Rd \times \Rd \to \Rs$ defined by $T'(u,w)=T(u,w)+S_wCu$ is skew-symmetric.

By the second identity of \eqref{eq:pi1AB}, the map $A(X,Y)$ satisfies the hypothesis of assertion~5 of
Lemma~\ref{l:mA}, so there exists $q \in \Rs$ such that for all $X, Y \in \Rs$,
\begin{equation}\label{eq:AXY}
    A(Y,X)=(\mA-\id)(X \wedge Y)q=\sum\nolimits_{i=0}^8 (\<S_i X,q\>S_i Y-\<S_i Y,q\>S_i X) - (\<X,q\> Y-\<Y,q\> X).
\end{equation}
Substituting (\ref{eq:BuXY}, \ref{eq:AXY}) into the fourth equation of \eqref{eq:pi1AB}, with $w=e_0$, we obtain
$\sum_{i=1}^8\<S_iq-\ve(T(e_0,e_i)+S_iC'e_0),Y\>S_iZ+\<2q+\ve(T(e_0,e_0)+(C-2C')e_0),Z\>Y=\ve S_{L^t_0Y}Z$, for all
$Y \in E_1(S_0), \, Z \in E_{-1}(S_0)$. Taking the inner product of the both sides with $S_kZ, \, k >0$, we get
$\<S_kq-\ve(T(e_0,e_k)+S_kC'e_0),Y\>\|Z\|^2+\<2q+\ve(T(e_0,e_0)+(C-2C')e_0),Z\>\<Y,S_kZ\>=\ve \<Y,L_0e_k\>\|Z\|^2$.
It follows that the second term on the right-hand side viewed as a polynomial of $Z \in E_{-1}(S_0)$
(with $Y\in E_1(S_0), \, k >0$ fixed) is divisible by $\|Z\|^2$. As this term is a product of two linear forms in $Z$,
we get $\<2q+\ve(T(e_0,e_0)+(C-2C')e_0),Z\>=0$. From the fact that $T(u,w)=T'(u,w)-S_wCu$, where $T'$ is
skew-symmetric, $\<q+\ve(C-C')e_0,Z\>=0$, for all $Z \in E_{-1}(S_0)$. As $e_0 \in \Rd$ is an arbitrary unit
vector, it follows that $\|u\|q+\ve(C-C')u \in E_{\|u\|}(S_u)$, for all $u \in \Rd$, so
$S_u(\|u\|q+\ve(C-C')u)=\|u\|(\|u\|q+\ve(C-C')u)$, which implies
$\|u\|(S_uq-\ve(C-C')u)=\|u\|^2q-\ve S_u(C-C')u$. Since $\|u\|$ is not a rational function, we obtain
$(C-C')u=\ve S_uq$, for all $u \in \Rd$, so $T(u,w)+S_wC'u=T'(u,w)-S_wCu+S_wC'u=T'(u,w)-\ve S_w S_uq=
T''(u,w)-\ve \<w,u\>q$, where $T''(u,w)=T'(u,w)-\ve (S_w S_u-\<w,u\>\id)q$ is skew-symmetric, as $T'$ is
skew-symmetric and by \eqref{eq:SiSj}. Substituting this to \eqref{eq:BuXY} we obtain
$B_u(Y,X)=\sum\nolimits_{i=0}^8 (\<T''(u,e_i)-\ve \<e_i,u\>q,X\>S_i Y-\<T''(u,e_i)-\ve \<e_i,u\>q, Y\>S_i X)
+\ve \<S_uq, X\>Y-\ve \<S_uq, Y\>X$. Substituting this expression and \eqref{eq:AXY} into \eqref{eq:op2bi} we get
after simplification: $\sigma_{XYZ}(2\sum_{i=0}^8 \<X,q\>S_i Y \wedge S_iZ-2\<X,q\> Y\wedge Z )=0$, so
$(\mA-\id)(\sigma_{XYZ}(\<X,q\> Y \wedge Z))=0$. From assertion~2 of Lemma~\ref{l:mA} it follows that
$\sigma_{XYZ}(\<X,q\> Y \wedge Z)=0$, for all $X,Y,Z \in \Rs$, which easily implies $q=0$. This proves the assertion
(with $T''$ denoted by $T$).

4. As it follows from \eqref{eq:abw} and assertion~\ref{it:2op2bi}, $\rho$ is a Codazzi tensor.
By \cite[Theorem~1]{DS}, for any two eigenspaces $E_\a, E_\b$ of $\rho$,
the exterior product $E_\a \wedge E_\b$ is an invariant subspace of the operator $R$ on the space of bivectors.
Suppose $\rho$ is not a multiple of the identity. As in the proof of Lemma~\ref{l:codazzi}, by linearity, it suffices
to show that the following assumption leads to a contradiction: there exist two orthogonal nontrivial invariant
subspaces $E_1, E_2$ of $\rho$ such that every $E_\a \wedge E_\b$ is an invariant subspace of $R$.
By \eqref{eq:op2ctc}, this is equivalent to the fact that every $E_\a \wedge E_\b$ is an invariant subspace of the
operator $\mA$, which is then equivalent to the fact that
\begin{equation}\label{eq:codazziop2}
\mA(X \wedge Y)Z=\sum\nolimits_{i=0}^8 (S_i X \wedge S_i Y) Z = 0,
\quad \text{for all } X, Y \in E_\a, \; Z \in E_\b, \quad \a \ne \b.
\end{equation}
Denote $\dim E_\a =d_\a > 0, \; \a=1,2$. By assertion~2 of Lemma~\ref{l:mA}, the space $\Sk(\Rs)$ is an invariant
subspace of $\mA$, and the restriction of $\mA$ to it is a symmetric operator, with eigenvalues $5$ and $-3$, whose
corresponding eigenspaces are $\mL_2$ and $\mL_3$. As every $E_\a \wedge E_\b$ is an invariant subspace of $\mA$, we
obtain that $E_\a \wedge E_\b= \oplus_{k=2,3} V_{\a\b k}$, where $V_{\a\b k}=\pi_{\mL_k} (E_\a \wedge E_\b)$. The six
subspaces $V_{\a\b k}, \; 1 \le \a \le \b \le 2$, $k=2,3$, are mutually orthogonal. Moreover, as
$\Sk(\Rs)=\oplus_{k=2,3} \mL_k = \oplus_{1 \le \a \le \b \le 2} E_\a \wedge E_\b$, we get
$\mL_k=\oplus_{1 \le \a \le \b \le 2} V_{\a\b k}$ (with all the direct sums above being orthogonal) and
$V_{\a\b k}= \pi_{E_\a \wedge E_\b} \mL_k =(E_\a \wedge E_\b) \cap \mL_k$, for all $1 \le \a \le \b \le 2, \; k=2,3$.

Every nonzero element $K \in V_{\a\a 2}$ has the form
$\sum_{i,j=0}^8 a_{ij}S_iS_j$, where $a_{ji}=-a_{ij}$. Moreover, as $V_{\a\a 2} \perp (E_\b \wedge \Rs),\; \b \ne \a$,
the kernel of every such $K$ contains $E_\b, \; \b \ne \a$. Choosing an orthonormal basis for $\Rd$, relative to
which the skew-symmetric matrix $a_{ij}$
has a canonical form, we get $K=\sum_{i=1}^4 b_i S_{2i-1}S_{2i}$, with $b_1 \ne 0$
(unless all the $b_i$'s are zeros). Then $X \in \Ker K$ if and only if $X=(\sum_{i=2}^4 c_i S_1S_2S_{2i-1}S_{2i})X$,
where $c_i=b_ib_1^{-1}, \; i=2,3,4$. Consider symmetric orthogonal operators
$D_i=S_1S_2S_{2i-1}S_{2i} \in \mL_4, \; i=2,3,4$. By \eqref{eq:SiSj}, $D_iD_j \in \mL_4, \; i \ne j$, and
$D_2 D_3 D_4 \in \mL_1$ (in fact, $D_2 D_3 D_4 = \pm S_0$). Then by assertion~1 of Lemma~\ref{l:mA},
$\Tr D_i =\Tr D_i D_j = \Tr D_2 D_3 D_4 =0, \; 2 \le i < j \le 4$. It follows that each of the symmetric orthogonal
operators $D_i, D_iD_j, D_2D_3D_4, \; 2 \le i < j \le 4$, has eigenvalues $\pm 1$, both of multiplicity $8$.
Furthermore, as the $D_i$'s pairwise commute (which again follows from \eqref{eq:SiSj}), we can choose an orthonormal
basis for $\Rs$ relative to which the matrices of the $D_i$'s are diagonal.
The $D_i$'s satisfy the above condition on the multiplicities of eigenvalues if and only if
the space $\Rs$ splits into the orthogonal sum of two-dimensional subspaces $W(\ve_2,\ve_3,\ve_4), \; \ve_i = \pm 1$,
such that $D_{i | W(\ve_2,\ve_3,\ve_4)}= \ve_i \id_{W(\ve_2,\ve_3,\ve_4)}$. From the above, $\Ker K$ is the
$+1$-eigenspace of the operator $c_2D_2+c_3D_3+c_4D_4$. Its eigenvalues are $\lambda=c_2\ve_2+c_3\ve_3+c_4\ve_4$,
with the corresponding eigenspaces $W_\lambda=\oplus W(\ve_2,\ve_3,\ve_4)$, where the sum is taken over the set
$s_\la=\{(\ve_2,\ve_3,\ve_4) \, : \, \lambda=c_2\ve_2+c_3\ve_3+c_4\ve_4\}$,
with $\dim V_\la= 2 \, \# s_\la$. Considering the equations $c_2\ve_2+c_3\ve_3+c_4\ve_4=1, \; \ve_i = \pm 1$, we see
that $\dim \Ker K$ can be equal to $0, 2, 4$, or $8$, and in the latter case (up to relabeling),
$c_2=\pm 1, \, c_2=c_3=0$, so $K$ is a nonzero multiple of $S_1S_2+c_2S_3S_4$ and $\Ker K$ is the $c_2$-eigenspace of
the symmetric operator $S_1S_2S_3S_4$. As
$\Ker K \supset E_\b, \; \dim\Ker K \ge d_\b$. It follows from $d_1+d_2=16$ that either one of the spaces $V_{\a\a 2}$
is trivial, or $d_1=d_2=8$ and $\Ker K = E_\b$, for all nonzero $K \in V_{\a\a 2}$ (both for $(\a,\b)=(1,2)$ and
$(\a,\b)=(2,1)$).

The first possibility leads to a contradiction. Indeed, suppose $V_{112}=0$. Then 
$\mL_2=V_{122}\oplus V_{222} \subset E_2 \wedge \Rs$ which implies that $\<KX,Y\>=0$, for all $K \in \mL_2$ and all
$X, Y \in E_1$, that is, $\<S_iX, S_jY\>=0$, for all $i,j=0, \dots, 8$. As for a nonzero $X$,
$\dim \Span_{i=0}^8(S_iX)=9$, 
it follows that $d_1=1$. Then for a nonzero $Z \in E_1$ we can choose $X \in E_2$ such that $Z \perp S_iX$. 
Substituting such $X, Z$ and an arbitrary $Y \in E_2$ into \eqref{eq:codazziop2} we find that $Z \perp S_iY$, for
all $Y \in E_2$. This implies that $E_2$ is an invariant subspace of all the operators $S_i$, hence $E_1$ also is. Then
$Z$ is an eigenvector of every $S_i$, which contradicts the fact that the operator $S_iS_j, \; i \ne j$, is orthogonal
and skew-symmetric.

Suppose now that $d_1=d_2=8$ and $\Ker K = E_\b$, for all nonzero $K \in V_{\a\a 2}$ (for $(\a,\b)=(1,2)$ and
$(\a,\b)=(2,1)$). Choose a nonzero $K \in V_{112}$. As it is shown above, under an appropriate choice of an
orthonormal basis for $\Rd, \; K =c(S_1S_2 +\ve S_3S_4)$ (for some $\ve = \pm1, \; c \ne 0$) and $E_2=\Ker K$ is
the $\ve$-eigenspace of $S_1S_2S_3S_4$. Then $E_1$ is the $(-\ve)$-eigenspace of $S_1S_2S_3S_4$.
As it follows from \eqref{eq:SiSj}, $S_1S_2S_3S_4S_i=\hat\ve_iS_iS_1S_2S_3S_4$, where $\hat\ve_i=-1$ for
$i=1, \dots, 4$ and $\hat\ve_i=1$ for $i=0,5, \dots, 8$. It follows that for any nonzero $X \in E_2, \; S_iX \in E_1$,
when $1\le i\le 4$ and $S_iX \in E_2$, otherwise. Moreover, as for any nonzero $X \in E_2, \; S_iS_jX =\pm S_kS_lX$,
where $\{i,j,k,l\}=\{1,2,3,4\}$,
the dimension of the space $\Span_{i,j=1}^4(S_iS_jX), \; X \in E_2$, is at most four, so there exists a nonzero
$Y \in E_2$ orthogonal to this subspace. Then $\Span_{i=1}^4(S_iX) \perp \Span_{i=1}^4(S_iY)$. Substituting such $X$
and $Y$ into \eqref{eq:codazziop2} and taking $Z=S_1X (\in E_1)$ we obtain $\|X\|^2 S_1Y=0$, which is a contradiction.

It follows that $\rho$ is a multiple of the identity, at every point $y \in \mU$. As $(\n_X \rho)Y=(\n_Y \rho)X$,
$\rho$ is in fact a constant multiple of the identity.

5. By assertion~1, there exists a linear map $N:\Rs \to \Sk(\Rs)$ such that
$\n_XS_w=[S_w,N_X]$. Then from \eqref{eq:abw} and assertion~\ref{it:2op2bi},
$[S_w,N_Y] X - [S_w,N_X]Y = \sum\nolimits_{i=0}^8 (\<T(w,e_i), X\> S_i Y - \<T(w,e_i), Y\> S_i X)$.
As from \eqref{eq:SiSj} $[S_w,\sum_{i,j=0}^8 \<T(e_i,e_j), X\> S_iS_j]=4 \sum_{i=0}^8 \<T(w,e_i), X\> S_i$,
for a linear map $N':\Rs \to \Sk(\Rs)$ defined by $N'_X=N_X+\frac14\sum_{i,j=0}^8 \<T(e_i,e_j), X\> S_iS_j$,
we obtain that $[S_w,N'_Y] X = [S_w,N'_X]Y$, for all $X, Y \in \Rs$. Taking a unit vector $w \in \Rd$ and
$X \in E_\ve(S_w), \; Y \in E_{-\ve}(S_w), \; \ve = \pm 1$, we get $[S_w,N'_Y] X =(S_w-\ve\id)N'_Y X =
-2\ve\pi_{E_{-\ve}(S_w)}N'_Y X$, and similarly, $[S_w,N'_X] Y =2\ve\pi_{E_{\ve}(S_w)}N'_X Y$, so
$\pi_{E_{-\ve}(S_w)}N'_Y X=\pi_{E_{\ve}(S_w)}N'_X Y$, hence $\<N'_Y X,Z\>=0$, for
all $X \in E_\ve(S_w), \; Y,Z \in E_{-\ve}(S_w)$. As $N'_Y$ depends linearly on $Y$ and is skew-symmetric, we obtain
that $\<N'_YX,Z\>=0$, for all $X \in E_1(S_w),\; Z\in E_{-1}(S_w)$ and all $Y \in \Rs$. It follows that
$E_1(S_w)$ and $E_{-1}(S_w)$ are invariant subspaces of $N'_Y$, so $N'_Y$ commutes with $S_w$, for any $w \in \Rd$
and any $Y \in \Rs$ (which, in fact, implies $N'=0$). Then
$\n_XS_w=[S_w,N_X]= [S_w,N'_X-\frac14\sum_{i,j=0}^8 \<T(e_i,e_j), X\> S_iS_j]=-\sum_{i=0}^8\<T(w,e_i), X\> S_i$, as
required.
\end{proof}


It follows from assertions~\ref{it:3op2bi} and \ref{it:4op2bi} of Lemma~\ref{l:op2bi} and \eqref{eq:op2nctc} that
after a conformal change of metric on
$\mU,\; (\n_ZR)(X,Y)=\ve \sum_{i,j=0}^8 (-\<T(e_i,e_j), Z\>S_jX \wedge S_i Y-\<T(e_i,e_j),Z\> S_iX \wedge S_j Y)=0$,
as $T$ is skew-symmetric. Hence $\mU$ is a locally symmetric space. Moreover, as $\rho$ is a constant multiple of the
identity by assertion~\ref{it:3op2bi} of Lemma~\ref{l:op2bi}, the curvature tensor \eqref{eq:op2ctc} is Osserman, so
$\mU$ is locally isometric to a rank-one symmetric space by \cite[Lemma~2.3]{GSV} (in fact, to the Cayley projective
plane or its noncompact dual, as these are the only two rank-one symmetric spaces of dimension $16$ the
Jacobi operator of whose curvature tensor has an eigenvalue of multiplicity exactly $8$).

Thus, for every $x \in M'$ satisfying assertion~\eqref{it:oc} of Lemma~\ref{l:locc1}, the metric on the neighborhood
$\mU=\mU(x)$ is conformally equivalent to the metric of a rank-one symmetric space.

\subsection{Proof of Theorem~\ref{t:co}}
\label{ss:glo}

Lemma~\ref{l:locc1} and the results of Sections~\ref{ss:cc} and \ref{ss:oc} imply the conformal part of
Theorem~\ref{t:co} at the generic points. Namely, every $x \in M'$ (the latter is an open, dense subset of $M^{16}$)
has a neighborhood $\mU$ which is conformally equivalent to a domain either of a Euclidean space, or of a
rank-one symmetric space, that is, of one of the model spaces
\begin{equation}\label{eq:models}
\Rs, \; \bc P^{8}, \; \bc H^{8}, \; \mathbb{H}P^{4}, \; \mathbb{H}H^{4}, \; \Oc P^2, \; \Oc H^2,
\end{equation}
where we normalize the standard metric $\tilde g$ on each of the non-flat spaces above in such a way that the
sectional curvature $K_\sigma$ satisfies $|K_\sigma| \in [1,4]$.

To prove the conformal part of Theorem~\ref{t:co}, we will show that, firstly, the same is true for any
$x \in M^{16}$, and secondly, that the model space to a domain of which $\mU$ is conformally equivalent is the same,
for all $x \in M^{16}$. Our proof very closely follows the arguments of \cite{Nco} from after Remark~4 to the end of
Section~3. We start with the following technical lemma:
\begin{lemma}\label{l:mmeps}
Let $(N^{16},\<\cdot,\cdot\>)$ be a smooth Riemannian space locally conformally equivalent to one of the
$\Oc P^2, \; \Oc H^2$, so that $\tilde g =f\<\cdot,\cdot\>$, for a positive smooth function
$f=e^{2\phi}: N^{16} \to \br$. Then the curvature tensor $R$ and the Weyl tensor $W$ of $(N^{16},\<\cdot,\cdot\>)$
satisfy (with $\ve=1$ for $\Oc P^2$ and $\ve=-1$ for $\Oc H^2$):
\begin{subequations}\label{eq:weylconfmodel}
\begin{align}
    R(X,Y)&=(X \wedge KY + KX \wedge Y) + \ve f (3X \wedge Y + P(X,Y)), \quad \text{where}\label{eq:model1} \\
    \notag
    P(X,Y)&=\mA(X \wedge Y)=\sum\nolimits_{i=0}^8 S_i X \wedge S_i Y, \quad
    K=H(\phi)-\n\phi \otimes \n\phi + \tfrac12 \|\nabla \phi\|^2 \id,\\
    W(X,Y)&=W_{\Oc,\ve}(X,Y)=\ve f (\tfrac{3}{5} X \wedge Y + P(X,Y)), \label{eq:model2} \\
    \|W\|^2&= \tfrac{32256}{5} f^2, \label{eq:model3} \\ 
    (\n_Z W)(X,Y)&=\ve Zf (\tfrac{3}{5} X \wedge Y + P(X,Y))  \label{eq:model4} \\
    \notag
    &+ \tfrac12 \ve([P(X,Y),\n f \wedge Z] + P((\n f \wedge Z)X,Y)+ P(X,(\n f \wedge Z)Y)),
\end{align}
\end{subequations}
where $H(\phi)$ is the symmetric operator associated to the Hessian of $\phi$, and both $\n$ and the norm are
computed relative to $\<\cdot,\cdot\>$.
\end{lemma}
\begin{proof}
By \eqref{eq:Rop2}, the curvature tensor of $\Oc P^2 \, (\ve=1)$ and $\Oc H^2 \, (\ve=-1)$ has
the form $\tilde R(X,Y)=\ve(3X \tilde \wedge Y + \sum_{i=0}^8 S_i X \tilde \wedge S_i Y)$,
where $(X \tilde\wedge Y)Z=\tilde g(X,Z)Y-\tilde g(Y,Z)X$. Under the conformal change of metric
$\tilde g =f\<\cdot,\cdot\>=e^{2\phi}\<\cdot,\cdot\>$, the curvature tensor transforms as
$\tilde R(X,Y)=R(X,Y)-(X \wedge KY + KX \wedge Y)$. As $\tilde g(X,Y) =f\<X,Y\>$, $X \tilde\wedge Y=f X \wedge Y$
and the $S_i$'s still satisfy \eqref{eq:SiSj} and are symmetric for $\<\cdot,\cdot\>$, equation
\eqref{eq:model1} follows.

Equation \eqref{eq:model2} follows from the definition of the Weyl tensor \eqref{eq:defWeyl}; the norm of $W$ in
\eqref{eq:model3} can be
computed directly using \eqref{eq:SiSj} and the fact that the $S_i$'s are symmetric, orthogonal and $\Tr S_i=0$.

Assertion~\ref{it:4op2bi} of Lemma~\ref{l:op2bi} is satisfied for $\tilde g$, so
$\tilde \n_Z S_i = \sum_{j=0}^8 \omega_{i}^{j}(Z) S_i$, where $\tilde \n$ is the Levi-Civita connection for $\tilde g$
and $\omega_{i}^{j}(Z)=-\<T(e_i,e_j), Z\>, \; \omega_{i}^{j}+\omega_{j}^{i}=0$. As
$\tilde \n_Z X = \n_Z X +Z\phi \, X + X\phi \, Z -\<X,Z\>\phi$, we get
$\n_Z S_i = \sum_{j=0}^8 \omega_{i}^{j}(Z) S_j + [S_i, \n \phi \wedge Z]$, so
$(\n_Z P)(X,Y)=[P(X,Y),\n \phi \wedge Z] + P((\n \phi \wedge Z)X,Y)+ P(X,(\n \phi \wedge Z)Y)$,
which, together with \eqref{eq:model2}, proves \eqref{eq:model4}.
\end{proof}

From Lemma~\ref{l:locc1}, the results of Sections~\ref{ss:cc} and \ref{ss:oc} and \cite[Theorem~3]{Nco},
for every point $x \in M'$, there exists a neighborhood $\mU=\mU(x)$ and a positive smooth function
$f: \mU \to \br$ such that the Riemannian space $(\mU, f \<\cdot,\cdot\>)$ is isometric to an open subset of one of
the model spaces of \eqref{eq:models}, so at every point $x \in M'$, the Weyl tensor $W$ of $M^{16}$ either vanishes,
or has the form $W_{\nu,\ve}$ given in \cite[Lemma~8, (36b)]{Nco}, with $n=16$ and the corresponding $\nu$, or
has the form $W_{\Oc,\ve}$ given in \eqref{eq:model2}. Here $W_{\nu,\ve}$ is the Weyl tensor of the corresponding
model space $M_{\nu,\ve}=\bc P^{8}, \; \bc H^{8}, \; \mathbb{H}P^{4}, \; \mathbb{H}H^{4}$, where $\ve=\pm 1$ is the
sign of the curvature and $\nu=1 (\nu=3)$ for complex (quaternionic) spaces, respectively.
The Jacobi operators associated to the different Weyl tensors $W_{\nu,\ve}$ of \cite[Eq.~(36b)]{Nco}
and $W_{\Oc,\ve}$ of \eqref{eq:model2} differ by the multiplicities and
the signs of the eigenvalues, so every point $x \in M'$ has a neighborhood conformally equivalent to a domain of
exactly one of the model spaces. Moreover, the function $f>0$ is well-defined when $W \ne 0$,
as $\|W\|^2 = C_{\nu n} f^2$ by \cite[Eq.~(36c)]{Nco} and $\|W\|^2= \tfrac{32256}{5} f^2$, by \eqref{eq:model3}.

By continuity, the Weyl tensor $W$ of $M^n$ either has the form $W_{\nu,\ve}$, or the form $W_{\Oc,\ve}$, or vanishes,
at every point $x \in M^n$ (as $M'$ is open and dense in $M^n$, see Lemma~\ref{l:locc1}). Moreover,
every point $x \in M^n$, at which the Weyl tensor has the form $W_{\nu,\ve}$ or $W_{\Oc,\ve}$, has a neighborhood, at
which the Weyl tensor has the same form. Hence $M^n=M_0 \sqcup \bigsqcup_\a M_\a$, where $M_0=\{x \, : \, W(x)=0\}$
is closed, and every $M_\a$ is a nonempty open connected subset, with $\partial M_\a \subset M_0$, such that the
Weyl tensor has the same form $W_{\nu,\ve}$ or $W_{\Oc,\ve}$ at every point $x \in M_\a$. In particular,
$M_\a \subset M'$, for every $\a$, so that each $M_\a$ is locally conformally equivalent to one of the nonflat model
spaces \eqref{eq:models}.

To prove the conformal part of Theorem~\ref{t:co}, we need to show that either $M = M_0$ or $M_0=\varnothing$. Suppose
that $M_0 \ne \varnothing$ and that there exists at least one component $M_\a$. If $M_\a$ is locally conformally
equivalent to one of the model spaces $M_{\nu,\ve}$, we get a contradiction following the arguments of \cite{Nco}
(from after Lemma~8 to the end of Section~3). Suppose $M_\a$ is locally conformally equivalent to one of
$\Oc P^2, \; \Oc H^2$.
Let $y \in \partial M_\a \subset M_0$ and let $B_\delta(y)$ be a small geodesic ball of $M$ centered at $y$ which is
strictly geodesically convex (any two points from $B(y)$ can be connected by a unique geodesic segment lying in
$B_\delta(y)$ and that segment realizes the distance between them). Let $x \in B_{\delta/3}(y) \cap M_\a$ and
let $r = \operatorname{dist}(x, M_0)$. Then the geodesic ball $B=B_r(x)$ lies in $M_\a$ and is strictly convex.
Moreover, $\partial B$ contains a point $x_0 \in M_0$. Replacing $x$ by the midpoint of the segment $[xx_0]$ and $r$
by $r/2$, if necessary, we can assume that all the points of $\partial B$, except for $x_0$, lie in $M_\a$.

The function $f$ is positive and smooth on $\overline{B} \setminus \{x_0\}$
(that is, on an open subset containing $\overline{B} \setminus \{x_0\}$, but not containing $x_0$).

\begin{lemma}\label{l:x_0}
When $x \to x_0,\; x \in B$, both $f$ and $\nabla f$ have a finite limit. Moreover, $\lim_{x\to x_0, x \in B}f(x)=0$.
\end{lemma}
\begin{proof}
The fact that $\lim_{x \to x_0, x \in B}f(x)=0$ follows from \eqref{eq:model3} and the fact that $W_{|x_0}=0$
(as $x_0 \in M_0$).

As the Riemannian space $(B, f\<\cdot,\cdot\>)$ is locally isometric to a rank-one symmetric space $M^{\Oc}_{\ve}$
(where $\ve= \pm 1$ and $M^{\Oc}_+=\Oc P^2, \; M^{\Oc}_-=\Oc H^2$) and
is simply connected, there exists a smooth isometric immersion $\iota:(B, f\<\cdot,\cdot\>) \to M^{\Oc}_{\ve}$. Since
$f$ is smooth on $\overline{B} \setminus \{x_0\}$ and $\lim_{x \to x_0, x \in B}f(x)=0$, the range of $\iota$ is
a bounded domain in $M^{\Oc}_{\ve}$. Moreover, as $\lim_{x \to x_0, x \in B}f(x)=0$, every sequence of points in $B$
converging to $x_0$ in the metric $\<\cdot,\cdot\>$ is a Cauchy sequence for the metric $f\<\cdot,\cdot\>)$. It
follows that there exists a limit $\lim_{x \to x_0, x \in B} \iota(x) \in M^{\Oc}_{\ve}$. Defining for every
$x \in B$ the point $\mS_{|x}=\Span_{i=0}^8(S_i)$ in the Grassmanian $G(9, \Sym(T_xM^{16}))$, we obtain
that there exists a limit $\lim_{x\to x_0, x \in B}\mS_{|x}=:\mS_{|x_0}\in G(9,\Sym(T_{x_0}M^{16}))$.
In particular, if $Z$ is a continuous vector field on $\overline{B}$, then there exists a unit continuous vector field
$Y$ on $\overline{B}$ such that $Y \perp \Span_{i=0}^8(S_i Z)$ on $B$. For such two vector fields, the function
$\theta(Y,Z)=\<\sum_{j=1}^{16} (\n_{E_j} W)(E_j,Y)Y,Z\>$ (where $E_j$ is an orthonormal frame on $\overline{B}$) is
well-defined and continuous on $\overline{B}$. Using (\ref{eq:SiSj}, \ref{eq:model4}) and assertion~4 of
Lemma~\ref{l:mA},  we obtain by a direct computation that at the points of $B$, $\theta(Y,Z)=\frac{52}{5}\ve Zf$. As
$\theta(Y,Z)$ is continuous on $\overline{B}$, there exists a limit $\lim_{x \to x_0, x \in B} Zf$. Since $Z$ is an
arbitrary continuous vector field on $\overline{B}$, $\n f$ has a finite limit when $x \to x_0, \; x \in B$.
\end{proof}


As $\lim_{x \to x_0, x \in B}f(x)=0$ and the $S_i$'s are orthogonal, the second term on the right-hand
side of \eqref{eq:model1} tends to $0$ when $x \to x_0$ in $B$.
Then the tensor field defined by $(X,Y) \to (X \wedge KY + KX \wedge Y)$ has a finite limit
(namely $R_{|x_0}$) when $x \to x_0$ in $B$. It follows that the symmetric operator $K$ has a finite limit at $x_0$.
Computing the trace of $K$ and using the fact that $\phi = \frac12 \ln f$ we get
\begin{equation*}
    \triangle u = F u , \quad \text{where} \; u=f^{7/2}, \; F =7 \Tr K
\end{equation*}
on $B$. Both functions $F$ and $u$ are smooth on $\overline{B} \setminus\{x_0\}$ and have a finite limit
at $x_0$. Moreover, $\lim_{x \to x_0, x \in B}u(x)=0$ by Lemma~\ref{l:x_0} and $u(x)>0$ for
$x \in \overline{B} \setminus\{x_0\}$. The domain $B$ is a small geodesic ball, so it satisfies
the inner sphere condition (the radii of curvature of the sphere $\partial B$ are uniformly bounded).
By the boundary point theorem \cite[Section~2.3]{Fra}, the inner directional derivative of $u$ at $x_0$
(which exists by Lemma~\ref{l:x_0}, if we define $u(x_0)=0$ by continuity) is positive. But
$\n u=\frac72 f^{5/2} \n f$ in $B$, so $\lim_{x \to x_0, x \in B} \n u=0$ by Lemma~\ref{l:x_0}, a contradiction.

This proves the conformal part of Theorem~\ref{t:co}.

The ``genuine" Osserman part, the Osserman Conjecture (assuming
Conjecture~A), now easily follows. Indeed, any Osserman manifold $M^n, \; n > 4$, is Einstein, hence by
\eqref{eq:defWeyl} its Weyl tensor is Osserman, hence $M^n$ is locally conformally equivalent to a rank-one symmetric
space or to a flat space, as shown above. Then by \cite[Theorem~4.4]{Nic},
$\n W =0$, so, as $M^n$ is Einstein, it is locally symmetric, and the proof follows from \cite[Lemma 2.3]{GSV}.


\begin{thebibliography}{BGNSt}

\bibitem[ABS]{ABS}
Atiah M.F., Bott R., Shapiro A.
\emph{Clifford modules},
Topology, \textbf{3, suppl.1} (1964), 3 -- 38.

\bibitem[Bes]{Bes}
Besse A.
\emph{Manifolds all of whose geodesics are closed}, (1978), Springer Verlag.

\bibitem[BG1]{BG1}
Bla\v zi\'c N., Gilkey P.
\emph{Conformally Osserman manifolds and conformally complex space forms},
Int. J. Geom. Methods Mod. Phys. \textbf{1} (2004), 97 -- 106.

\bibitem[BG2]{BG2}
Bla\v zi\'c N., Gilkey P.
\emph{Conformally Osserman manifolds and self-duality in Riemannian geometry}.
Differential geometry and its applications, 15--18, Matfyzpress, Prague, 2005.

\bibitem[BGNSi]{BGNSi}
Bla\v zi\'c N., Gilkey P., Nik\v cevi\'c S., Simon U.
\emph{The spectral geometry of the Weyl conformal tensor}.
PDEs, submanifolds and affine differential geometry, 195--203, Banach Center Publ.,
\textbf{69}, Polish Acad. Sci., Warsaw, 2005.

\bibitem[BGNSt]{BGNSt}
Bla\v zi\'c N., Gilkey P., Nik\v cevi\'c S., Stavrov I.
\emph{Curvature structure of self-dual $4$-manifolds},
arXiv: math.DG/ 0808.2799.

\bibitem[BG]{BG}
Brown R., Gray A.
\emph{Riemannian manifolds with holonomy group $Spin(9)$},
Diff. Geom. in honor of K.Yano, Kinokuniya, Tokyo (1972), 41 -- 59.

\bibitem[BGV]{BGV}
Brozos-V\'{a}zquez M., Garc\'{i}a-Río E., V\'{a}zquez-Lorenzo R.
\emph{Osserman and Conformally Osserman Manifolds with Warped and Twisted Product Structure},
Results Math. \textbf{52} (2008), 211 -- 221.

\bibitem[Chi]{Chi}
Chi Q.-S.
\emph{A curvature characterization of certain locally rank-one symmetric spaces},
J. Differ. Geom. \textbf{28}(1988), 187 -- 202.

\bibitem[Der]{Der}
Derdzinski A.
\emph{Exemples de m\'{e}triques de K\"{a}hler et d'Einstein auto-duales sur le plan complexe}.
G\'{e}om\'{e}trie riemannienne en dimension $4$ (S\'{e}minaire Arthur Besse 1978/79).
Cedic/Fernand Nathan, Paris (1981), 334-346.

\bibitem[DS]{DS}
Derdzinski A., Shen C.-L.
\emph{Codazzi tensor fields, curvature and Pontryagin forms},
Proc. London Math. Soc. (3), \textbf{47} (1983), 15 -- 26.

\bibitem[Fra]{Fra}
Fraenkel L. E.
\emph{An introduction to maximum principles and symmetry in elliptic problems},
Cambridge Tracts in Mathematics, 128. Cambridge University Press, Cambridge, 2000.

\bibitem[Gil]{Gil}
Gilkey P. \emph{The geometry of curvature homogeneous pseudo-Riemannian manifolds}.
ICP Advanced Texts in Mathematics, 2. Imperial College Press, London, 2007.

\bibitem[GSV]{GSV}
Gilkey P., Swann A., Vanhecke L.
\emph{Isoparametric geodesic spheres and a conjecture of Osserman concerning the Jacobi operator},
Quart. J. Math. Oxford (2), \textbf{46}(1995), 299 -- 320.

\bibitem[Hus]{Hus}
Husemoller D. \emph{Fiber bundles}, (1975), Springer-Verlag.

\bibitem[Nag]{Nag}
Nagata M.
\emph{A remark on the unique factorization theorem}, J. Math. Soc. Japan, \textbf{9} (1957), 143--145.

\bibitem[Nic]{Nic}
Nickerson H. K.
\emph{On conformally symmetric spaces}, Geom. Dedicata \textbf{18} (1985), 87--99.

\bibitem[N1]{Nhjm}
Nikolayevsky Y.
\emph{Osserman manifolds and Clifford structures},
Houston J. Math. \textbf{29}(2003), 59--75.

\bibitem[N2]{Nmm}
Nikolayevsky Y.
\emph{Osserman manifolds of dimension $8$},
Manuscripta Math. \textbf{115}(2004), 31 -- 53.

\bibitem[N3]{Nma}
Nikolayevsky Y.
\emph{Osserman Conjecture in dimension $n \ne 8, 16$},
Math. Ann. \textbf{331}(2005), 505 -- 522.

\bibitem[N4]{Nbel}
Nikolayevsky Y.
\emph{On Osserman manifolds of dimension $16$},
Contemporary Geometry and Related Topics, Proc. Conf. Belgrade, 2005 (2006),
379 -- 398.

\bibitem[N5]{Nco}
Nikolayevsky Y.
\emph{Conformally Osserman manifolds}, arXiv: 0810.5621. To appear in Pacific J. Math.

\bibitem[Oss]{O}
Osserman R.
\emph{Curvature in the eighties},
Amer. Math. Monthly,
\textbf{97}(1990), 731 -- 756.

\bibitem[TV]{TV}
Tricerri F., Vanhecke L. \emph{Curvature homogeneous Riemannian manifolds.}
Ann. Sci. École Norm. Sup. (4) \textbf{22} (1989), 535--554.

\end{thebibliography}
\end{document}